\documentclass[a4paper,11pt,leqno]{article}

\usepackage[latin1]{inputenc}
\usepackage[T1]{fontenc} 
\usepackage[english]{babel}
\usepackage{verbatim}
\usepackage{calligra}
\usepackage{enumerate}
\usepackage{dsfont}
\usepackage{amsfonts}
\usepackage{amsmath}
\usepackage{amsthm}
\usepackage{amssymb}
\usepackage{mathtools}
\usepackage{mathrsfs}
\usepackage{color}
\usepackage{graphicx}
\usepackage{bm}
\usepackage{srcltx}

\usepackage{tikz}
\usepackage[retainorgcmds]{IEEEtrantools}

\usepackage{graphicx}		
\usepackage[hypcap=false]{caption}

\theoremstyle{definition}
\newtheorem{defin}{Definition}[section]

\newtheorem{rmk}[defin]{Remark}

\theoremstyle{plane}
\newtheorem{thm}[defin]{Theorem}
\newtheorem{prop}[defin]{Proposition}
\newtheorem{cor}[defin]{Corollary}
\newtheorem{lemma}[defin]{Lemma}

\newcommand{\tsl}{\textsl}

\newcommand{\mc}{\mathcal}
\newcommand{\mf}{\mathfrak}

\newcommand{\veps}{\varepsilon}
\newcommand{\what}{\widehat}
\newcommand{\wtilde}{\widetilde}
\newcommand{\vphi}{\varphi}

\renewcommand{\k}{\kappa}
\newcommand{\s}{\sigma}
\renewcommand{\t}{\tau}
\newcommand{\z}{\zeta}

\newcommand{\de}{\delta}
\renewcommand{\o}{\omega}

\newcommand{\R}{\mathbb{R}}

\newcommand{\N}{\mathbb{N}}
\newcommand{\Z}{\mathbb{Z}}
\newcommand{\T}{\mathbb{T}}

\renewcommand{\div}{{\rm div}\,}
\newcommand{\divh}{{\rm div}_{\!h}}

\newcommand{\curl}{{\rm curl}\,}

\newcommand{\supp}{{\rm supp}\,}

\newcommand{\dx}{ \, {\rm d} x}
\newcommand{\dt}{ \, {\rm d} t}
\newcommand{\B}{B^s_{p, r}}

\newcommand{\weak}{{\rm w}}
\newcommand{\loc}{{\rm loc}}

\allowdisplaybreaks

\def\d{\partial}
\def\div{{\rm div}\,}
\newcommand{\dd}{{\rm d}}

\newcommand{\fra}[1]{\textcolor{blue}{[[FF: #1]]}}

\begin{document}

\title{\textsc{\Large{\textbf{Local and global strong solutions to a reduced model \\ for inviscid micropolar fluids}}}}

\author{\normalsize \textsl{Francesco Fanelli}$\,^{1,2,3},$
\textsl{Pedro Gabriel Fern\'andez-Dalgo}$\,^{1},$ 
\textsl{Mar\'ia Eugenia Mart\'inez Martini}$\,^{4}$
\vspace{.5cm} \\
\footnotesize{$\,^{1}\;$ \textsc{BCAM -- Basque Center for Applied Mathematics}} \\ 
{\footnotesize Alameda de Mazarredo 14, E-48009 Bilbao, Basque Country, SPAIN} \vspace{.2cm} \\
\footnotesize{$\,^{2}\;$ \textsc{Ikerbasque -- Basque Foundation for Science}} \\  
{\footnotesize Plaza Euskadi 5, E-48009 Bilbao, Basque Country, SPAIN} \vspace{.2cm} \\
\footnotesize{$\,^{3}\;$ \textsc{Universit\'e Claude Bernard Lyon 1}, {\it Institut Camille Jordan -- UMR 5208}} \\ 
{\footnotesize 43 blvd. du 11 novembre 1918, F-69622 Villeurbanne cedex, FRANCE} \vspace{.2cm} \\
\footnotesize{$\,^{4}\;$ \textsc{Universidad de Chile}, {\it Center for Mathematical Modeling}} \\ 
{\footnotesize Casilla 170, Correo 3, Santiago, CHILE} \vspace{.3cm} \\
%
%
\footnotesize{Email addresses: \ttfamily{ffanelli@bcamath.org}}, $\;$
\footnotesize{\ttfamily{pfernandez@bcamath.org}}, $\;$
\footnotesize{\ttfamily{maria.martinez.m@uchile.cl}}
\vspace{.2cm}
}

\date\today

\maketitle

\subsubsection*{Abstract}
{\footnotesize 

This paper investigates the well-posedness issue for a reduced two-dimensional model of micropolar fluids.
This reduced model presents a coupling between an Euler-type equation for the velocity field of the fluid and a transport-diffusion equation for the microrotation
field (which is a scalar field, in this setting).

We establish the local existence and uniqueness of strong solutions in the scale of Besov space $B^s_{p,1}$ having regularity index $s\geq1+2/p$.
Furthermore, in the subcritical case when $s>1+2/p$, we prove that these solutions exist globally in time. The global persistence of regularity in the critical
setting $s=1+2/p$ remains open.
}

\paragraph*{\small 2020 Mathematics Subject Classification:}{\footnotesize 35Q35 
(primary);
35A01, 
35B65, 
76B03, 
35B45  
(secondary).}

\paragraph*{\small Keywords: }{\footnotesize micropolar fluids; inviscid fluids; strong solutions; global well-posedness; critical regularity.
}

%
%

\section{Introduction and main results} \label{s:intro}

Micropolar fluids are complex fluids equipped of microstructure. They consist of rigid,
either spherical or randomly oriented, particles suspended in a medium, where the deformation of fluid particles is ignored \cite{Luk}.
These internal particles can undergo independent rotations, leading to non-symmetric stress tensors and a strong coupling between macroscopic velocity and
microscopic angular momentum.

At the mathematical level, the system of micropolar fluid was introduced in Eringen \cite{Erin}, where the author provides a natural modification
to the classical Navier-Stokes equations obtained by incorporating the internal microstructure of the fluid through an additional microrotation field.
In this framework, besides the macroscopic velocity field $U(t,x) \in \R^3$ and pressure field $\pi(t,x) \in \R$ of the fluid, one considers an additional vector field,
describing precisely the above mentioned independent rotation of microscopic particles, denoted here by $M(t,x)\in \R^3$.
This should not be confused with the rotational of $U$.

Following the derivation given in \cite{Luk}, the dynamics of a general, three-dimensional, homogeneous micropolar fluid can be described by a non-linear system
of equations, encoding basic physical conservation principles, as detailed below.
\begin{itemize}
\item[(i)] Incompressibility: 
\[
\div U \,=\,0\,.
\]
\item [(ii)] Conservation of linear momentum:
\[
\partial_t U\, +\, (U \cdot \nabla) U\, +\, \nabla \pi\, -\, \left( \mu + \mu_r\right) \,\Delta U\, =\, 2\, \mu_r\, \curl M\,.
\]
\item[(iii)] Conservation of angular momentum: 
\[
\partial_t M\, +\, (U \cdot \nabla )M\, -\, (c_a+c_d)\,\Delta M\,-\,\left( c_0 + c_d-c_a\right)\, \nabla \div M\,+\, 4\,\mu_r\, M\, =\, 2 \,\mu_r\, \curl U\,.
\]
\end{itemize}
In the equations above, $\mu\ge 0$ denotes the dynamic viscosity coefficient of the fluid, $\mu_r\ge 0$ represents the dynamic microrotation viscosity
and $c_0$, $c_d$, $c_a$ are coefficients of angular viscosities. 

The micropolar fluid model is particularly suitable for describing complex fluids such as suspensions, polymers, or biological fluids, where the internal structure
significantly influences the macroscopic behavior. The model has also been used in the study of porous media flows and lubrication theory, as well as in
engineering applications involving microfluidics and materials with internal structure (see \tsl{e.g.} \cite{Erin}, \cite{Erin2} and \cite{Luk} for comprehensive discussions).

The analysis of this model has developed in parallel with that of the classical Navier-Stokes system. Similar to it,
Galdi and Rionero \cite{GaldiRionero1977} proved the existence of global in time finite energy weak solutions for $3$-D micropolar fluid equations.
We refer to \cite{Luk0, Luk1, Luk, Guterres, Rojas-Medar} for further results in this direction. Strong solutions are typically defined locally in time
(see \cite{Boldrini0, Boldrini, Luk1, Ortega, Rojas-Medar0} for results in various functional settings), but become global
under suitable smallness assumptions (see \tsl{e.g.} \cite{Cruz} for $H^1\times H^1$ solutions and
\cite{Chen2012} for the case of scaling-critical spaces $\dot B^{3/p-1}_{p,1}$).
However, as in the classical case, the question of global regularity for large data in three space dimensions remains open. 
Thus, regularity criteria have been proposed, in the spirit of the celebrated Prodi-Serrin criterion 
for the classical Navier-Stokes equations. For instance, Dong and Chen \cite{DongChen} proved the first Serrin-type regularity criteria under the condition
$u\in L^q\big([0, T], L^{p}\big)$, $2/q+3/p\le 1$, improved later to Lorentz spaces $L_x^{p, \infty}$ by \cite{Yuan} and to the scaling critical case $2/q+3/p= 1$ by
\cite{Loayza}.
See also \cite{OrtegaTorres, Ragusa} for further results in this direction.

\subsection{Two-dimensional reductions} \label{ss:reductions}

Owing to the complexity of the resulting system of equations, two-dimensional reductions have been proposed in the literature, as they look
better suited for mathematical analysis and simulations.
In order to introduce them, we first need some preparation.

For vectors $A\in \R^3$, consider the decomposition $A= (A_h, A^3)$, with $A_h=(A^1, A^2)\in \R^2$ and $A^3\in \R$. One then formulates the \tsl{ansatz}
of flat flux, in the sense that one restricts the attention to special solutions of the form
\[
U\,=\,\big(u_h(t, x_h), 0\big)\qquad \mbox{ and }\qquad M\,=\,\big(0, 0, m(t, x_h)\big)\,, \qquad\qquad \mbox{ with }\qquad \divh u_h\,=\,0\,.
\]
Here, we have set $\divh u_h\,=\,\d_1u^1\,+\,\d_2u^2$ to be the classical two-dimensional divergence operator.
In particular, with the above \tsl{ansatz}, one obtains $\div M=0$, $\curl M\,=\,\big(\partial_2m, -\partial_1 m, 0\big)$ and $\curl U=(0, 0 , \omega)$,
where $\o\,=\,\d_1u^2\,-\,\d_2u^1$ is the two-dimensional vorticity of the fluid.

With this special form of solutions, and substituting $u_h=u$, the system above becomes the following 2-D reduced model:
\begin{equation}\label{eq:2D-microp-gen}
\begin{cases}
\partial_t u\, +\, (u \cdot \nabla) u \,+\, \nabla \pi\, -\, \left( \mu + \mu_r\right)\, \Delta u \,=\, -\,2\, \mu_r\, \nabla^\perp m \\[1ex]
\partial_t m\, +\, u \cdot \nabla m\,-\,(c_a+c_d)\,\Delta m\, +\, 4\,\mu_r\, m\, =\, 2\, \mu_r\, \omega \\[1ex]
\div u =0\,.
\end{cases}
\end{equation}
It goes without saying that all the operators appearing in system \eqref{eq:2D-microp-gen} are now two-dimensional differential operators, acting only
with respect to the horizontal variables $x_h\in\R^2$. 

\medbreak
As a matter of fact, the theory significantly improve for the $2$-D reduction \eqref{eq:2D-microp-gen}, with respect to the general three-dimensional system.
For instance, strong solutions $(u,m)\in \mc C\big(\R_+; H^1(\R^2)\big) \cap L_{\loc}^2\big(\R_+; H^2(\R^2)\big)$
are known to exist globally in time; see \tsl{e.g.} to \cite{DongChen}, where decay to solutions of the heat equation is shown under suitable assumptions
on the initial data.
In \cite{DongZhang}, the authors constructed global strong solutions to equations \eqref{eq:2D-microp-gen} \emph{without angular viscosity} (that is,
for the choice $c_a+c_d=0$, $\mu+\mu_r>0$). This result was later extended and refined in \cite{DongLi} by  obtaining global well-posedness and decay estimates under
partial dissipation. 
In this context, more recently, Brandolese, Busuioc, Iftime
and Perusato \cite{B-B-I-P} investigated the role of viscosity parameters in large-time behavior of global finite energy weak solutions to the two-dimensional
system.
In \cite{Xue}, Xue established global well-posedness for the reduced model \eqref{eq:2D-microp-gen}
within the framework of inhomogeneous Besov spaces.
A study of low regularity solutions, including vortex sheets type solutions, was recently performed in \cite{BL-C-S} by
B\'ejar-L\'opez, Cunha and Soler.

We remark that all the above mentioned works kept the viscosity in the velocity equation and additionally assumed some dissipation mechanism
(either microrotation viscosity, or at least damping) in the equation of conservation of angular momentum, that is the equation for $m$.

\medbreak
In the present paper, instead, we want to devote attention to the ``dual'' situation in which we assume the fluid to be inviscid
(that is, $\mu+\mu_r=0$) and we impose absence of any damping mechanism for the microrotation field $m$, while we keep the effect
of the microrotation viscosity (namely, $c_a+c_d>0$).
We detail in Subsection \ref{ss:equations} below the precise system of equations we are going to study throughout this work.

The goal of this work is to continue the investigation carried out in \cite{F-FD}, where global existence and uniqueness of solutions \tsl{\`a la Yudovich} was established,
and to prove global persistence of higher regularities. 

As a matter of fact, we have recently discovered that the system of equations under study here was introduced even before \cite{F-FD}, and
already appeared in the work \cite{Jiu} by Jiu, Liu, Wu and Yu.
Therein, the authors studied the well-posedness issue in a smooth bounded domain of $\R^2$. Compared to our study, they worked in the framework
of Sobolev spaces $H^s$ and imposed higher regularity hypotheses on the initial datum $\big(u_0,m_0\big)$ than those we will assume here.

\subsection{The equations under study} \label{ss:equations}

Let $u=(u^1,u^2)\in\R^2$ denote the velocity field of the fluid, and let $\omega$ be the associated scalar vorticity, defined by
\[
 \o\,:=\,\d_1u^2\,-\,\d_2u^1\,. 
\]
Let the scalar field $\Pi\in\R$ represent the fluid pressure and $m\in\R$ the microrotation field, encoding the microstructure of the fluid particles.
The system of equations under consideration in this work is the following one:
\begin{equation} \label{eq:2D-microp}
\left\{\begin{array}{l}
        \d_tu\,+\,(u\cdot\nabla)u\,+\,\nabla\Pi\,=\,-\,\alpha\,\nabla^\perp m \\[1ex]
        \d_tm\,+\,u\cdot\nabla m\,-\,\k\,\Delta m\,=\,\alpha\,\o \\[1ex]
        \div u\,=\,0\,.
       \end{array}
\right.
\end{equation}
Here, the parameters $\alpha>0$ and $\kappa>0$ are fixed positive constants, and we adopt the notation
\[
\nabla^\perp m \,:= (-\partial_2 m,\partial_1 m).
\]
All the (scalar and vector) fields appearing in equations \eqref{eq:2D-microp} are depending on the time variable $t$ and the space variable $x$.
Throughout this work, we consider
\[
 (t,x)\,\in\,\R_+\times\Omega\,,\qquad \mbox{ with }\qquad \Omega\,=\,\R^2\ \mbox{ or }\ \T^2\,,
\]
and system \eqref{eq:2D-microp} is supplemented with the initial data
\[
 \big(u,m\big)_{|t=0}\,=\,\big(u_0,m_0\big)\,,
\]
for which we will state precise assumptions in Subsection \ref{ss:results}.

Notice that, compared to \eqref{eq:2D-microp-gen}, the above reduced system \eqref{eq:2D-microp} completely removes viscous diffusion effect from the momentum equation
(that is, the equation for the velocity field $u$),
as well as any damping mechanism from the microrotation field $m$.

\medbreak
The main objective of this work is to establish a (possibly global) well-posedness theory for the inviscid micropolar system \eqref{eq:2D-microp},
in the framework of strong solutions and with no restriction on the size of the initial datum.
Besides its intrinsic interest, such a theory provides a natural framework for the approximation of the Yudovich type solutions constructed in \cite{F-FD}.

Here, we will look for solutions possessing somehow \emph{minimal regularity}. In this respect, we point out that, owing to the intricate
coupling of the equations, system \eqref{eq:2D-microp} does not really possess a suitable scaling.
Instead, in order to find a convenient functional framework for studying well-posedness, one has to remark that system \eqref{eq:2D-microp},
and especially the equation for $u$, presents an underlying transport structure by the velocity field $u$ itself. This requires
at least $u\in L^1_T(W^{1,\infty})$ in order to guarantee persistence of the initial smoothness.
At the same time, propagating the Lipschitz norm of $u$ requires the right-hand side to be in the same space; thus one also needs, roughly speaking,
$\nabla m$ to have the same space regularity as $u$. In turn,
this combines well with the structure of the equation for $m$: its right-hand side depends on $\omega\,=\,\curl(u)$, so it is
one derivative less regular than $u$, but, thanks to parabolic regularisation, one may expect to gain two orders of smoothness on $m$, implying that
$\nabla m$ possesses indeed the same regularity as $u$.

In this work, we decide to focus on the framework of Besov spaces $B^s_{p,1}$, for both a sake of generality and mathematical
robustness. For simplicity, we decide to focus on spaces with third index equal to $1$,
in order to avoid technical complications coming from
the study of parabolic equations in general $B^s_{p,r}$ spaces; however, we expect that it is possible to handle the general case in a similar manner.

To conclude, let us mention that, in the argument depicted above, one should pay attention to the endpoint case in which $p=+\infty$,
as singular integrals are involved in our study. Here we will make use of our choice of working in Besov spaces, as in this setting
one can prove a sort of maximal regularity effect (gain of two full derivatives with respect of the right-hand side) for parabolic equations
also in the case $B^s_{\infty,r}$. At the same time, in order to treat the velocity field and its vorticity, for which no smoothing
effect can be expected, we will need to impose an additional finite integrability assumption
either on $u_0$ or on $\o_0\,:=\,\curl(u_0)$.

\subsection{Statement of the main results} \label{ss:results}

As already mentioned, in the present paper we address the question of well-posedness of the reduced micropolar fluid system \eqref{eq:2D-microp},
in the framework of Besov spaces of type $B^s_{p,1}$, under suitable assumptions
on $p$ and $s$. For simplicity of presentation and analysis, we only focus on the case of third Besov index equal to $1$,
but we expect that similar results hold true also for general $r\in [1,+\infty]$ in the subcritical case $s\,>\,1+2/p$.

Let us first consider the case of a finite integrability index $p\in\,]1,+\infty[\,$. 

\begin{thm} \label{th:global-strong}
Let $p\in\,]1,+\infty[\,$ and let $s\,\geq\,1\,+\,2/p$. Take an initial datum $\big(u_0,m_0\big)$ such that
\[
 \div u_0\,=\,0\,,\qquad u_0\,\in\,B^s_{p,1}(\R^2)\qquad \mbox{ and }\qquad m_0\,\in\,B^{s-1}_{p,1}(\R^2)\,.
\]

Then, there exists a time $T>0$ and a solution $\big(u,m\big)$ to system \eqref{eq:2D-microp} related to the initial datum $\big(u_0,m_0\big)$, such that
\[
 u\,\in\,\mc C\big([0,T];B^s_{p,1}(\R^2)\big)\qquad \mbox{ and }\qquad m\,\in\,\mc C\big([0,T];B^{s-1}_{p,1}(\R^2)\big)\,\cap\,L^1\big([0,T];B^{s+1}_{p,1}(\R^2)\big)\,.
\]
Moreover, the solution is global if $s>1+2/p$, under one of the following additional assumptions:
\begin{enumerate}[\rm (G.1)]
 \item either $1<p\leq2$,
 \item or $2<p<+\infty$, if one assumes further that  $u_0$ and $m_0$ belong to $L^2(\R^2)$;
 \item or $2<p<+\infty$, if one assumes further that there exists $p_0\in\,]1,2[\,$ such that $\o_0$ and $m_0$ belong to $L^{p_0}(\R^2)$,
where we have defined  $\o_0\,:=\,\curl(u_0)\,=\,\d_1u_0^2\,-\,\d_2u_0^1$ to be the vorticity of the initial velocity $u_0$.
\end{enumerate}
Under either condition {\rm (G.1)} or condition {\rm (G.2)}, one further obtains that
\[
 u\,\in\,\mc C\big(\R_+;L^2(\R^2)\big)\qquad \mbox{ and }\qquad m\,\in\,\mc C\big(\R_+;L^2(\R^2)\big)\,\cap\,L^2_\loc\big(\R_+;H^1(\R^2)\big)\,.
\]
Under condition {\rm (G.3)}, instead, after defining the vorticity $\o\,:=\,\curl(u)\,=\,\d_1u^2\,-\,\d_2u^1$, one also gets
\[
 \o\,,\,m \;\in\,\mc C\big(\R_+;L^{p_0}(\R^2)\big)\,.
\]
Finally, the previous (local or global) solution is unique in its functional class in any of the following cases:
\begin{enumerate}[\rm (U.1)]
 \item if $s\,\geq\, 2\,\big(1\,+\,1/p\big)$;
 \item if $1\,+\,2/p\,\leq\,s\,<\,2\,+\,2/p$ and one further assumes any of the conditions {\rm (G.1)}, {\rm (G.2)} or {\rm (G.3)}.
\end{enumerate}

\end{thm}

It must be noted that Theorem \ref{th:global-strong} fails to prove global existence in the critical case $s=1+2/p$. This happens also for the case $p=+\infty$ (see
also Remark \ref{rmk:Yudovich} below). We refer to Subsection \ref{rmkCritical} for more details about the issues arising in the critical setting. 

\medbreak
Next, we consider the endpoint case $p=+\infty$. In this situation, we need an extra integrability assumption on the initial vorticity field $\o_0$
and on the initial microrotation field $m_0$. It goes without saying that this assumption could be replaced by a finite energy hypothesis on the initial
datum $\big(u_0,m_0\big)$, in the spirit of condition (G.2) above. However, we omit the details here and simply focus in the
somehow more difficult case.

\begin{thm} \label{th:endpoint}
Let $s\,\geq\,1$. Take an initial datum $\big(u_0,m_0\big)$ such that
\[
 \div u_0\,=\,0\,,\qquad u_0\,\in\,B^s_{\infty,1}(\R^2)\qquad \mbox{ and }\qquad m_0\,\in\,B^{s-1}_{\infty,1}(\R^2)\,.
\]
Assume moreover that there exists $p_0\in\,]1,2[\,$ such that both $\o_0$ and $m_0$ belong to $L^{p_0}(\R^2)$,
where we have defined  $\o_0\,:=\,\curl(u_0)\,=\,\d_1u_0^2\,-\,\d_2u_0^1$ to be the vorticity of the initial velocity $u_0$.

Then, there exists a time  $T>0$ and  a unique solution $\big(u,m\big)$ to system \eqref{eq:2D-microp} related to the initial datum $\big(u_0,m_0\big)$, such that
\begin{align*}
& u\,\in\,\mc C\big([0, T];B^s_{\infty,1}(\R^2)\big)\,, 
\qquad \mbox{ with }\quad \o\,\in\, \mc C\big([0, T];L^{p_0}(\R^2)\big)\,, \\
&\qquad \mbox{ and }\qquad\qquad
m\,\in\,\mc C\big([0, T];B^{s-1}_{\infty,1}(\R^2)\big)\,\cap\,\mc C\big([0, T];L^{p_0}(\R^2)\big)\,\cap\,
L^1\big([0, T];B^{s+1}_{\infty,1}(\R^2)\big)\,.
\end{align*}
In addition, if $s>1$, the solution is global. 
\end{thm}

Let us conclude this part with some remarks about the previous statement.

\begin{rmk} \label{r:p=inf}
By Besov embeddings, it follows from the assumptions of Theorem \ref{th:endpoint} that $\o_0$ and $m_0$ also belong to $L^\infty(\R^2)$. Next, 
it is well-known (see \tsl{e.g.} \cite{F-FD, Maj-Bert}) that, if $\o_0\in L^{p_0}(\R^2)\cap L^\infty(\R^2)$, for $p_0\in\,]1,2[\,$,
then the corresponding velocity field $u_0$, recasted from $\o_0$ by solving the Biot-Savart law (recalled in \eqref{eq:BS} below),
satisfies $u_0\in L^{q_0}(\R^2)\cap L^\infty(\R^2)$, where $q_0\in\,]2,+\infty[\,$ is given by the relation $1/q_0\,=\,1/p_0\,-\,1/2$.

It would not be difficult to replace the assumption
$\o_0,m_0\,\in\,L^{p_0}(\R^2)$ with the assumption $u_0,m_0\,\in\,L^2(\R^2)$, and get a similar well-posedness statement (global in the case $s>1$).
We refer to the end of Subsection \ref{ss:leb} for more comments in this respect.
%
\end{rmk}

\begin{rmk}\label{rmk:Yudovich}
Paper \cite{F-FD} proves gobal existence and uniqueness of Yudovich-type solutions for system \eqref{eq:2D-microp}. Roughly speaking, this correspond to the case
$s=1$ in Theorem \ref{th:endpoint}, but in the larger space $B^1_{\infty, \infty}$, instead of the critical one $B^1_{\infty, 1}$,
for the velocity field. In addition, that result fails to recover a bound on two derivatives for the variable $m$.

We refer to Subsection \ref{rmkCritical} for more details on the critical case $s=1$. 
\end{rmk}

\subsection*{Organisation of the paper}

After this introduction, we now give an overview of the rest of the paper.

In Section \ref{s:LP}, we begin by recalling the main ingredients from Littlewood-Paley theory that will be needed throughout the paper.
Section \ref{s:a-priori} is devoted to the derivation of the key a priori estimates underlying our analysis. These estimates play a central role in the paper and will later allow us, in Section \ref{s:uniqueness}, to establish a general stability result at the regularity level of Theorems \ref{th:global-strong} and \ref{th:endpoint},
together with uniqueness of solutions. 
We then turn to the question of existence. In Section \ref{s:existence}, we prove the local existence of strong solutions, while Section \ref{s:global} is devoted to showing that these solutions are in fact global whenever $s>1+2/p$.

\subsection*{Notation} 

Let us fix some notation which will be freely used throughout this paper.

Given a Banach space $\mf B$ over $\R^2$, we will adopt the same notation $\mf B(\R^2)$ for scalar, vector-valued and matrix valued functions.
Very often we will simply write $\mf B$, as no confusion can arise in this paper.
Typically, we will resort to the longer notation $\mf B(\R^2)$ when formulating the assumptions and the statements, and in
important centered formulas.

For an interval $I\subset \R$ and $\mf B$ as above,
we denote by $\mc C\big(I;\mf B\big)$ the space of continuous functions on $I$ with values in $\mf B$. For any $p\in[1,+\infty]$,
the symbol $L^p\big(I;\mf B\big)$ stands for the space of measurable functions on $I$ such that the map $t\mapsto \left\|f(t)\right\|_{\mf B}$ belongs to $L^p(I)$.
When $I=[0,T]$, we will often use the shorten notation $L^r_T(\mf B)\,=\,L^r\big([0,T];\mf B\big)$, whereas we will resort to the full notation
in statements and centered formulas.

In our estimates we will often avoid to write the explicit multiplicative constants which allow to pass from one line to the other. Thus, we will write
$A\,\lesssim\, B$ meaning that there exists a universal constant $C>0$, not depending on the solutions nor on the data (in the latter case, this will be pointed out),
such that $A\,\leq\,C\,B$. We will write $A\approx B$ if $A\lesssim B$ and $B\lesssim A$.



We define the operator $\nabla^\perp$ as $\nabla^\perp\,=\,\big(-\d_2,\d_1\big)$.
Finally, given a two-dimensional vector field $v\,=\,\big(v^1,v^2\big)$, we define its $\curl$ as
$\curl(v)\,=\,\d_1v^2\,-\,\d_2v^1$.

\section*{Acknowledgements}

{\small

This work has been partially supported 
by the Basque Government through the BERC 2022-2025 program and by the Spanish State Research Agency through the BCAM Severo Ochoa excellence accreditation
CEX2021-001142.

The first author also acknowledges the support of the European Union through the COFUND program [HORIZON-MSCA-2022-COFUND-101126600-SmartBRAIN3],
and of the the French National Research Agency (ANR) through the project CRISIS (ANR-20-CE40-0020-01).

M.E.M.M's work was partly funded by Chilean research grant Centro de Modelamiento Matem\'atico (CMM) BASAL fund FB210005 for center of excellence from ANID-Chile.
}

\section{Elements of Littlewood-Paley theory} \label{s:LP}

The present section is intended to be a functional toolbox, where we collect all the material which is needed in the course of the proof of the main results.
We start by recalling, in Subsection \ref{ss:LP-Sobolev}, some basic facts of Littlewood-Paley theory and
the definition of non-homogeneous Besov spaces. Subsection \ref{ss:para} is devoted to the presentation
of results from paradifferential calculus, namely paraproduct decomposition and estimates for composition of Besov functions.
Finally, in Subsection \ref{ss:tools-est} we recall some well-known estimates for smooth solutions to transport and transport-diffusion equations
in Besov spaces.

\subsection{Littlewood-Paley decomposition and Besov spaces} \label{ss:LP-Sobolev}

We recall here the main ideas of Littlewood-Paley theory in $\R^d$, where, for the sake of generality, $d\geq1$.
If not otherwise specified, we refer to Chapter 2 of \cite{BCD} for details.

Fix a smooth radial function $\chi$ supported in the ball $B(0,2)$, equal to $1$ in a neighborhood of $B(0,1)$
and such that $r\mapsto\chi(r\,e)$ is nonincreasing over $\R_+$ for all unitary vectors $e\in\R^d$. Set
$\varphi\left(\xi\right)=\chi\left(\xi\right)-\chi\left(2\xi\right)$ and
$\vphi_j(\xi):=\vphi(2^{-j}\xi)$ for all $j\geq0$.
The (non-homogeneous) dyadic blocks $(\Delta_j)_{j\in\Z}$ are defined by\footnote{Hereafter, we agree  that  $f(D)$ stands for 
the pseudo-differential operator $u\mapsto\mc{F}^{-1}[f(\xi)\,\what u(\xi)]$, where $\mc F u = \what u$ denotes the Fourier transform of $u$
and $\mc F^{-1}$ the inverse Fourier transform.} 
$$
\Delta_j\,:=\,0\quad\mbox{ if }\; j\leq-2,\qquad\Delta_{-1}\,:=\,\chi(D)\qquad\mbox{ and }\qquad
\Delta_j\,:=\,\varphi(2^{-j}D)\quad \mbox{ if }\;  j\geq0\,.
$$
We  also introduce the following low frequency cut-off operator:
\begin{equation} \label{eq:S_j}
S_ju\,:=\,\chi(2^{-j}D)\,=\,\sum_{k\leq j-1}\Delta_{k}\qquad\mbox{ for }\qquad j\geq0\,.
\end{equation}
Note that $S_j$ and $\Delta_j$ are convolution operators by $L^1$ kernels, whose norms are independent of the index $j$. Thus, they act continuously
from $L^p$ into itself, for any $p\in[1,+\infty]$.

Remark that the function $\chi$ can be chosen so that, for all $\xi\in\R^d$, one has the equality $\chi(\xi)+\sum_{j\geq0}\vphi_j(\xi)=1$.
Based on this partition of unity on the Fourier space, we obtain the so-called non-homogeneous \emph{Littlewood-Paley decomposition} of tempered distributions:
\[
\forall\,u\,\in\,\mc S'(\R^d)\,,\qquad\qquad\qquad u\,=\,\sum_{j\geq-1}\Delta_ju\qquad \mbox{ in the sense of }\quad \mc S'\,.
\]

Spectrally localised functions have very nice properties under the action of differentiation operators. These properties are known as
\emph{Bernstein inequalities}, which we now state. 

\begin{lemma} \label{l:bern}
Let  $0<r<R$.   A constant $C>0$ exists so that, for any integer $k\geq0$, any couple $(p,q)$ 
in $[1,+\infty]^2$, with  $p\leq q$,  and any function $u\in L^p$, for all $\lambda>0$ we have:
\begin{align*}
{\supp}\, \widehat u\, \subset\,   B(0,\lambda R)\,=\,\big\{\xi\in\R^d\,\big|\,|\xi|\leq\lambda R \big\}\qquad
\Longrightarrow\qquad
\|\nabla^k u\|_{L^q}\, \leq\,
 C^{k+1}\,\lambda^{k+d\left(\frac{1}{p}-\frac{1}{q}\right)}\,\|u\|_{L^p}\,, \\[1ex]
{\supp}\, \widehat u   \, \subset\, \big\{\xi\in\R^d\,\big|\, r\lambda\leq|\xi|\leq R\lambda\big\}
\quad\Longrightarrow\quad C^{-k-1}\,\lambda^k\|u\|_{L^p}\,\leq\,
\|\nabla^k u\|_{L^p}\,
\leq\,C^{k+1} \, \lambda^k\|u\|_{L^p}\,.
\end{align*}
\end{lemma}

After these preliminaries, we can now define the class of Besov spaces.
\begin{defin} \label{d:B}
  Let $s\in\R$ and $1\leq p,r\leq+\infty$. The \emph{non-homogeneous Besov space}
$B^{s}_{p,r}\,=\,B^s_{p,r}(\R^d)$ is defined as the set of tempered distributions $u\in \mc S'(\R^d)$ for which
$$
\|u\|_{B^{s}_{p,r}}\,:=\,
\left\|\left(2^{js}\,\|\Delta_ju\|_{L^p}\right)_{j\geq -1}\right\|_{\ell^r}\,<\,+\infty\,.
$$
\end{defin}
Besov spaces are interpolation spaces between the more classical Sobolev spaces. In fact, for any $k\in\N$ and~$p\in[1,+\infty]$,
we have the following chain of continuous embeddings:
\begin{equation}\label{eq:BesovLp}
B^k_{p,1}\hookrightarrow W^{k,p}\hookrightarrow B^k_{p,\infty}\,,
\end{equation}
where  $W^{k,p}$ stands for the classical Sobolev space of $L^p$ functions with all the derivatives up to the order $k$ in $L^p$.
As a matter of fact, when $1<p<+\infty$, one has
$B^k_{p, \min (p, 2)}\hookrightarrow W^{k,p}\hookrightarrow B^k_{p, \max(p, 2)}$ (see Theorems 2.40 and 2.41 in \cite{BCD}).
In particular, for all $s\in\R$, we deduce the equivalence $B^s_{2,2}\equiv H^s$, with equivalence of norms. 
In addition, for any $k\in\N$ and any $\veps\in\,]0,1[\,$, one has the equivalence $B^{k+\veps}_{\infty,\infty}\equiv \mc C^{k,\veps}$
with the classical H\"older spaces, with equivalence of norms.

As an immediate consequence of the first Bernstein inequality, one gets the following embedding result.
\begin{prop}\label{p:embed}
The space $B^{s_1}_{p_1,r_1}$ is continuously embedded in the space $B^{s_2}_{p_2,r_2}$ for all indices satisfying $p_1\,\leq\,p_2$ and
$$
s_2\,<\,s_1-d\left(\frac{1}{p_1}-\frac{1}{p_2}\right)\qquad\qquad\mbox{ or }\qquad\qquad
s_2\,=\,s_1-d\left(\frac{1}{p_1}-\frac{1}{p_2}\right)\;\;\mbox{ and }\;\;r_1\,\leq\,r_2\,. 
$$
\end{prop}

We underline here that, in particular, we get the following property: for any triplet $(s, p, r) \in \mathbb{R} \times [1, +\infty]^2$
such that 
\begin{equation} \label{eq:AnnLInfty}
s > \frac{d}{p} \qquad\qquad \text{ or } \qquad\qquad s = \frac{d}{p}\quad \text{ and }\quad r = 1\,,
\end{equation}
one has the chain of continuous embeddings 
\begin{equation*}
B^s_{p,r}(\R^d) \hookrightarrow B^{s - \frac{d}{p}}_{\infty, r}(\R^d) \hookrightarrow B^0_{\infty, 1}(\R^d) \hookrightarrow L^\infty(\R^d)\,.
\end{equation*}
Similarly, if \eqref{eq:AnnLInfty} is reinforced to the condition
\begin{equation} \label{cond:Lipschitz}
s\,>\,1\,+\,\frac{d}{p}\qquad\qquad\quad \mbox{ or }\qquad\qquad\quad s\,=\,1\,+\,\frac{d}{p}\quad \mbox{ and }\quad r\,=\,1\,,
\end{equation}
one has the embedding
\[
 B^s_{p,r}(\R^d)\,\hookrightarrow\,W^{1,\infty}(\R^d)\,,
\]
which plays a key role in our study (keep in mind the discussion in the Introduction).

To conclude this part, we recall a useful inequality of Gagliardo-Nirenberg type.
It can be shown by means of Littlewood-Paley decomposition: its proof consists in cutting a tempered distribution into low and high frequencies,
estimating them differently by using the two Bernstein inequalities and, finally, optimising the frequency size $N$ at which realising the cut.
Since this argument is quite standard, we omit the details of the proof.

\begin{lemma} \label{l:GN-ineq} 
For any function $ f \in \mathcal{S}(\R^d) $, for any $p\in\,]1,+\infty[\,$ the following Gagliardo-Nirenberg type inequality holds true:
 \[
\| f \|_{L^\infty}\,\lesssim\,\| f\|_{L^p}^{\alpha}\,\|\nabla f \|^{1-\alpha}_{L^\infty}\,,
\]
where the exponent $ \alpha\,=\,\alpha(p,d)\; \in \,]0,1[\, $ is defined by $\alpha = p/(p+d)$, whence $1-\alpha = d/(p+d)$.

The above inequalities extend also to vector-valued functions. In addition, for any vector field $ u \in \mathcal{S}(\R^d)  $ such that $ \div(u) = 0 $, one has
\[
\|u\|_{L^\infty}\,\lesssim\,\left\|u\right\|_{L^p}^{\alpha}\,\left\|\Omega\right\|_{L^{\infty}}^{1-\alpha}\,,
\]
 where  $ \Omega\,:=\,Du\,-\,^t\nabla u $ denotes the vorticity matrix of $ u $ and $ \alpha \in \,]0,1[\, $ is as above.
\end{lemma}

\subsection{Paradifferential calculus} \label{ss:para}

We now apply Littlewood-Paley decomposition to state some useful results from paradifferential calculus. Again, we refer to Chapter 2
of \cite{BCD} for details and further results.

To begin with, let us introduce the \emph{paraproduct operator} (after J.-M. Bony \cite{Bony}) and paraproduct decomposition.
For this, we observe that, 
formally, we can decompose any product of two tempered distributions $u$ and $v$ into the sum
\begin{equation}\label{eq:bony}
u\,v\;=\;\mathcal{T}_uv\,+\,\mathcal{T}_vu\,+\,\mathcal{R}(u,v)\,,
\end{equation}
where we have defined the paraproduct operator $\mc T$ and the reminder operator $\mc R$ as
$$
\mathcal{T}_uv\,:=\,\sum_jS_{j-1}u\,\Delta_j v,\qquad\qquad\mbox{ and }\qquad\qquad
\mathcal{R}(u,v)\,:=\,\sum_j\sum_{|k-j|\leq1}\Delta_j u\,\Delta_{k}v\,.
$$

The paraproduct and remainder operators have many nice continuity properties over Besov spaces. We collect them in the next statement.
\begin{prop}\label{p:op}
For any $(s,p,r)\in\R\times[1,+\infty]^2$ and $\s>0$, the paraproduct operator 
$\mathcal{T}$ maps continuously $L^\infty\times B^s_{p,r}$ in $B^s_{p,r}$ and  $B^{-\s}_{\infty,\infty}\times B^s_{p,r}$ in $B^{s-\s}_{p,r}$.
Moreover, the following estimates hold:
\[
\|\mathcal{T}_uv\|_{B^s_{p,r}}\,\lesssim\,\|u\|_{L^\infty}\,\|\mc P(D) v\|_{B^{s-1}_{p,r}}\qquad \mbox{ and }\qquad
\|\mathcal{T}_uv\|_{B^{s-\s}_{p,r}}\,\lesssim\,\|u\|_{B^{-\s}_{\infty,\infty}}\,\|\mc P(D) v\|_{B^{s-1}_{p,r}}\,,
\]
where $\mc P(D)$ is a Fourier multiplier\footnote{For instance, one may have $\mc P(D) = \nabla$, or $\mc P(D) = \curl$.} of order $1$.

For any $(s_1,p_1,r_1)$ and $(s_2,p_2,r_2)$ in $\R\times[1,+\infty]^2$ such that 
$s_1+s_2>0$, $1/p:=1/p_1+1/p_2\leq1$ and~$1/r:=1/r_1+1/r_2\leq1$,
the remainder operator $\mathcal{R}$ maps continuously~$B^{s_1}_{p_1,r_1}\times B^{s_2}_{p_2,r_2}$ into~$B^{s_1+s_2}_{p,r}$.
In the case $s_1+s_2=0$, if in addition $r=1$, the operator $\mathcal{R}$ is continuous from $B^{s_1}_{p_1,r_1}\times B^{s_2}_{p_2,r_2}$ with values
in $B^{0}_{p,\infty}$.
\end{prop}

The consequence of this proposition is that, when $s>0$, the product acts continuously on the spaces $B^s_{p,r}\cap L^\infty$;
in particular, the product is continuous from the space $B^s_{p,r}\times B^s_{p,r}$ into $B^s_{p,r}$
as long as condition \eqref{eq:AnnLInfty} holds with $s > 0$.
Moreover, in that case, we have the so-called \emph{tame estimates}.

\begin{cor}\label{c:tame}
Let $(s,p, r)\in\R\times[1,+\infty]^2$ be such that that $s > 0$. Then, for any $f$ and $g$ belonging to $L^\infty\cap B^s_{p,r}$, the product
$fg$ also belongs to that space and we have
\begin{equation}
\label{alg:prop:2}
\left\| f\,g \right\|_{B^s_{p,r}}\, \lesssim \,\| f \|_{L^\infty}\, \|g\|_{B^s_{p,r}}\, +\, \| f \|_{B^s_{p,r}} \,\| g \|_{L^\infty}\,.
\end{equation}
\end{cor}

We stress the fact that the product does \emph{not} act continuously on the spaces $B^0_{p,r}$, as the continuity of the remainder operator $\mc R(f,g)$ with
values in $B^0_{p,r}$ fails in this case. 

\medbreak
Finally, we need a few commutator estimates in our study. Recall that, given two operators $A$ and $B$, we denote by $[A,B]\,:=\,AB-BA$ their commutator operator.
The first commutator estimate is classical and corresponds to Lemma 2.100 of \cite{BCD} (see also Remark 2.101 therein),
to which we also add a small variant of it (see \cite{Cobb-F} for details).
\begin{lemma}\label{l:CommBCD}
Let $s>0$ and $(s,p,r)\in\R\times[1,+\infty]^2$ such that condition \eqref{cond:Lipschitz} is satisfied. 
Let $v$ be a vector field on $\R^d$, belonging to $B^s_{p,r}$. 
Then we have
\begin{equation*}
\forall\, f \in \B\,, \qquad\qquad  2^{js}\left\| \big[ v \cdot \nabla, \Delta_j \big] f  \right\|_{L^p}\, \lesssim\, c_j\,
\Big( \|\nabla v \|_{L^\infty} \| f \|_{\B} + \|\nabla v \|_{B^{s-1}_{p, r}} \|\nabla f \|_{L^\infty} \Big)\,,
\end{equation*}
and also
\begin{equation*} 
\forall\, f \in B^{s-1}_{p, r}\,, \qquad
2^{j(s-1)} \left\| \big[ v \cdot \nabla, \Delta_j \big] f  \right\|_{L^p}\, \lesssim\, c_j\, \Big( \|\nabla v \|_{L^\infty} \| f \|_{B^{s-1}_{p, r}} +
\|\mc P(D) v \|_{B^{s-1}_{p, r}} \| f \|_{L^\infty} \Big)\,,
\end{equation*}
where the $\big(c_j\big)_{j\geq -1}$ are (possibly distinct) sequences in the unit ball of $\ell^r$ and where, in the second
inequality, $\mc P(D)$ is a Fourier multiplier of order $1$.
\end{lemma}

The second commutator estimate we need involves the commutator of a paraproduct operator with a Fourier multiplier. It is contained in Lemma 2.99 of \cite{BCD},
but we report it here for the reader convenience.
\begin{lemma}\label{l:ParaComm}
Let $\psi$ be a smooth function on $\mathbb{R}^d$, which is homogeneous of degree $m$ away from a neighbourhood of $\,0$.
Then, for any vector field $v$ such that $\nabla v \in L^\infty$ and for any $s\in\R$, one has:
\begin{equation*}
\forall\, f \in B^s_{p,r}\,, \qquad\qquad \left\| \big[ \mathcal{T}_v, \psi(D) \big] f \right\|_{B^{s-m+1}_{p,r}}\, \lesssim\,
\|\nabla v\|_{L^\infty} \|f\|_{B^s_{p,r}}\,.
\end{equation*}
\end{lemma}

\subsection{Time-dependent Besov spaces} \label{ss:Ch-Ler}
In the present subsection, we extend the above construction to time-dependent Besov spaces, often named \emph{Chemin-Lerner spaces}
(as first introduced in \cite{Chem-Ler}).

They naturally arise, and reveal to be particularly useful, when solving transport-diffusion equations
in Besov spaces (see Paragraph \ref{sss:transp-diff} below).
Indeed, when solving evolutionary PDEs, it is natural to use
spaces of type $L^q_T\big(X\big)$, with $X$ denoting some Banach space. However, when $X$ is a Besov space as in our case,
one has to localise first the equations by Littlewood-Paley decomposition and make estimates for each dyadic block. This will provide one
with  estimates of the $L^q_T\big(L^p\big)$ norm of  each dyadic block $\Delta_j$ \emph{before} performing the $\ell^r$ summation over $j$.
This  leads to the following definition (see \cite{Chem-Ler}, \cite{BCD}).
\begin{defin}\label{def:Besov,tilde}
Let $s\in \R$, $(q,p,r)\in [1,+\infty]^3$ and $T\in [0,+\infty]$. We set
$$
\|u\|_{\wtilde L^q_T(B^s_{p,r})}
\,:=\,\left\| \Bigl(  2^{js}  \|\Delta_j u(t)\|_{L^q_T(L^p)} \Bigr)_{j\geq -1}\right\|_{\ell^r}\,.
$$
We also set $\wtilde{\mc C}_T\big(B^s_{p,r}\big)\,:=\,\wtilde L_T^\infty\big(B^s_{p,r}\big)\cap \mc C\big([0,T];B^s_{p,r}\big)$.
\end{defin}

The relation between this new class of spaces and the classical spaces $L^q_T\big(B^s_{p,r}\big)$ can be easily recovered by use of the Minkowski
inequality: we have
\begin{equation} \label{est:Chem-Ler}
\|u\|_{\wtilde{L}^q_T(B^s_{p,r})}\;\leq\;\|u\|_{L^q_T(B^s_{p,r})}\quad \mbox{ if }\ q\,\leq\,r\,,\qquad\qquad
\|u\|_{\wtilde{L}^q_T(B^s_{p,r})}\;\geq\;\|u\|_{L^q_T(B^s_{p,r})} \quad \mbox{ if }\ q\,\geq\,r\,.
\end{equation}
Observe that we have equality of the two norms, so of the spaces, whenever $q=r$. In particular, we have
\[
 \wtilde{L}^1_T\big(B^s_{p,1}\big)\,=\,L^1_T\big(B^s_{p,1}\big)\,.
\]
This, together with the material of Paragraph \ref{sss:transp-diff}, motivates our choice of working in this class $B^s_{p,1}$. The generalisation of our
results to general $B^s_{p,r}$ spaces works similarly, at the price of additional technical complications, which we want to avoid here for the sake of
better clarity.

\medbreak
Time-dependent Besov spaces behave well with respect to product and paralinearisation operations. For instance, concerning
products, Bony's decomposition \eqref{eq:bony} and Proposition \ref{p:op} imply the following statement.
\begin{cor}\label{c:op_time}
Let $s>0$ and $(p,r,q_1,q_2,q_3,q_4)\in[1,+\infty]^6$ such that
\[
\frac 1q\,:=\,\frac{1}{q_1}\,  +\,\frac{1}{q_2}\,=\,\frac{1}{q_3}\, +\,\frac{1}{q_4}\,,\qquad\qquad\mbox{ with }\qquad q\in[1,+\infty]. 
\]

Then, there exists a constant $C>0$, depending only on $(d, s, p, r)$, such that 
$$
\|uv\|_{\wtilde L^q_T(B^s_{p,r})}\,
\leq\, C\,\left(\|u\|_{L^{q_1}_T(L^\infty)}\,\|v\|_{\wtilde L^{q_2}_T(B^{s}_{p,r})}
\,+\, \|u\|_{\wtilde L^{q_3}_T(B^{s}_{p,r})}\, \| v\|_{L^{q_4}_T(L^\infty)}\right)\,.
$$
\end{cor}

We avoid to present here equivalent versions in Chemin-Lerner spaces of the commutator estimates stated above, as they will not be needed in our study.
We only point out that time-dependent Besov spaces behave well also with respect to interpolation: we have
the inequality
\begin{align}
\label{est:CL-interp}
\left\|u\right\|_{\wtilde L^q_T(\B)}\,&\leq\,\left\|u\right\|_{\wtilde L^{q_1}_T(B^{s_1}_{p,r})}^\beta\;\left\|u\right\|_{\wtilde L^{q_2}_T(B^{s_2}_{p,r})}^{1-\beta} \\
\nonumber
&\qquad \qquad
\mbox{ provided }\quad \frac{1}{q}\,=\,\frac{\beta}{q_1}\,+\,\frac{1-\beta}{q_2}\quad \mbox{ and }\quad s\,=\,\beta\,s_1\,+\,(1-\beta)\,s_2\,.
\end{align}

\subsection{Transport and transport-diffusion equations in Besov spaces} \label{ss:tools-est}

We now apply Littlewood-Paley theory to the study of some evolution equations in Besov spaces.
In fact, system \eqref{eq:2D-microp} presents a hyperbolic-parabolic coupling, where the hyperbolic part is given by the Euler-type equation
for $u$ and the parabolic part concerns the transport-diffusion equation for $m$.
Therefore, it is natural to consider the linear problem
\begin{equation}\label{eq:TV}
\left\{\begin{array}{l}
\partial_t f\, +\, v \cdot \nabla f\,-\,\nu\,\Delta f\, = \,g \\[1ex]
f_{|t = 0}\, =\, f_0\,,
\end{array}\right.
\end{equation}
where $f_0$, $g$ and $v$ are suitable data of the problem and where the ``diffusion coefficient'' $\nu$ satisfies $\nu\geq0$.
Throughout this part, the velocity field $v=v(t,x)$ will always be assumed to be 
Lipschitz continuous with respect to the space variable.

In this subsection we present several \tsl{a priori} estimates in Besov spaces for solutions $f$ to the Cauchy problem \eqref{eq:TV},
in both cases in which $\nu>0$ (transport-diffusion equation) and $\nu=0$ (transport equation).
Most of the material presented here can be found in Chapter 3 of \cite{BCD}.

\subsubsection{Transport estimates} \label{sss:transport}
We first discuss the case $\nu=0$, in which the equation in \eqref{eq:TV} reduces to a classical transport equation.
The following statement contains the conclusions of Theorems 3.14 and 3.19 of \cite{BCD}, in the case when the transport field $v$ is Lipschitz-continuous.
Here below, given a Banach space $X$, we denote by $\mc C^0_\weak\big([0,T];X\big)$ the space of functions which are continuous in time with values in
$X$ endowed with its weak topology.

\begin{thm}\label{th:transport}
Let $(s,p,r)\in\R\times[1,+\infty]^2$ satisfy the Lipschitz condition \eqref{cond:Lipschitz} and $T>0$ be fixed.
Assume that $v$ is a 
vector field belonging to $L^1\big([0,T];B^s_{p,r}\big)$ such that, for some
$q > 1$ and $M > 0$, $v \in L^q\big([0,T];B^{-M}_{\infty, \infty}\big)$.
Finally, let $\s\leq s$ be such that
\[
\s\,\geq\,-\,d\,\min\left\{\frac{1}{p}\,,\,\frac{1}{p'}\right\}\qquad\qquad \mbox{ or, \ \ \ \ \  if }\ \div v=0\,,\quad 
\s\,\geq\,-\,1\,-\,d\,\min\left\{\frac{1}{p}\,,\,\frac{1}{p'}\right\}\,.
\]

Then, for any external force $g \in L^1\big([0,T];B^\s_{p,r}\big)$ and any initial datum $f_0 \in B^\s_{p,r}$, the transport equation \eqref{eq:TV} has a unique solution $f$ in the space:
\begin{itemize}
\item $\mc C\big([0,T];B^\s_{p,r}\big)$ if $r < +\infty$;
\item $\left( \bigcap_{\s'<\s} \mc C\big([0,T];B^{\s'}_{p,\infty}\big) \right) \cap \mc C_{\weak}\big([0,T];B^\s_{p, \infty}\big)$, in the case $r = +\infty$.
\end{itemize}
Moreover, after defining the function $V(t)$ as
\[
V(t)\,:=\,\int^t_0\left\|\nabla v(\t)\right\|_{B^{s-1}_{p,r}}\,\dd\t\,,
\]
the unique solution $f$ satisfies the following estimate, for a suitable universal constant $C>0$:
\begin{equation*} 
\forall\,t\in[0,T]\,,\qquad
\| f(t) \|_{B^\s_{p,r}}\, \leq\, e^{C\,V(t)}\,\left(\| f_0 \|_{B^\s_{p,r}} + \int_0^t e^{-C\,V(\tau)}\, \| g(\t) \|_{B^\s_{p,r}} \dd\t\right)\,.
\end{equation*}
In the case when $v=f$, the previous estimate holds true with $V'(t)\,=\,\left\|\nabla f(t)\right\|_{L^\infty}$.
\end{thm}

\subsubsection{Estimates for transport-diffusion equations} \label{sss:transp-diff}

We now focus on the case $\nu>0$ in \eqref{eq:TV}. In this situation, the classical estimates in Besov spaces (which naturally involve the use
of Chemin-Lerner spaces) present an exponential growth in the Lipschitz norm of the transport field $v$; we refer to Theorem 3.37 of \cite{BCD} for details.

However, such an exponential growth is not convenient for our scopes. Owing to the difference of regularity between $u$ and $m$ in system
\eqref{eq:2D-microp}, it is much more convenient to look at the transport term $u\cdot\nabla m$ as a forcing term.
Therefore, here we replace equation \eqref{eq:TV} with the forced heat equation
\begin{equation}\label{eq:heat}
\left\{\begin{array}{l}
\partial_t f\,-\,\nu\,\Delta f\, = \,h \\[1ex]
f_{|t = 0}\, =\, f_0\,,
\end{array}\right.
\end{equation}
where, now, the forcing term $h$ also includes the transport term. The next statement is simply a formalisation of estimate (3.39) of \cite{BCD}.

\begin{thm} \label{t:tr-diff}
Let $\nu>0$, $T>0$ and $\s\in\R$ be fixed. Take an initial datum $f_0\in B^\s_{p,r}$ and an external force
$h\in\wtilde L^1_T\big(B^\s_{p,r}\big)$.

Then, there exists a constant $C>0$, depending only on the regularity index $\s$, such that, for any smooth solution
$f$ of equation \eqref{eq:heat} and any $q\in[1,+\infty]$ and any $q_1\in[1,q]$, one has
\begin{align*}
\nu^{1/q}\,\left\|f\right\|_{\wtilde L^q_T(B^{\s+2/q}_{p,r})}\,\leq\,C\,\big(1+\nu T\big)^{\max\{1/q, 1+1/q-1/q_1\}}\,
\left(\left\|f_0\right\|_{B^\s_{p,r}}\,+\,\nu^{-1+1/q_1}\,\left\|h\right\|_{\wtilde L^{q_1}_T(B^{\s-2+2/q_1}_{p,r})}\right)\,.
\end{align*}
In particular, one deduces that
\begin{align*}
\left\|f\right\|_{\wtilde L^\infty_T(B^{\s}_{p,r})}\,+\,\nu\,\left\|f\right\|_{\wtilde L^1_T(B^{\s+2}_{p,r})}\,\leq\,C\,\big(1\,+\,\nu\,T\big)\,
\left(\left\|f_0\right\|_{B^\s_{p,r}}\,+\,\left\|h\right\|_{\wtilde L^1_T(B^{\s}_{p,r})}\right)\,.
\end{align*}
\end{thm}

\section{\tsl{A priori} estimates} \label{s:a-priori}
In this section, we perform \tsl{a priori} estimates for solutions to system \eqref{eq:2D-microp}. Then, given an initial datum
$\big(u_0,m_0\big)$ as in the statement of Theorem \ref{th:global-strong} or \ref{th:endpoint}, we assume to dispose of a smooth solution $\big(u,m\big)$
defined on $\R_+\times\R^2$, decaying fast enough at infinity, and we perform formal computations to bound its norm in some suitable functional spaces.

In Subsection \ref{ss:leb} we will bound $L^p$ norms of $u$, $m$ and the vorticity $\omega$ of the fluid. In Subsection \ref{ss:bes},
instead, we will apply the estimates of Subsection \ref{ss:tools-est} to get a bound on their Besov norms.

\subsection{Bounds for the Lebesgue norms} \label{ss:leb}
Here, we get \tsl{a priori} estimates for suitable Lebesgue norms of the solution. We start the discussion by considering the vorticity $\o$ and
the microrotation field $m$, then we tackle the estimates for the velocity field $u$.

\subsubsection{The vorticity and the microrotation field} \label{ss:Leb_o-m}
To begin with, we compute an equation for the vorticity $\o\,=\,\curl(u)\,=\,\d_1u^2\,-\,\d_2u^1$. By applying the $\curl$ operator to the equation for $u$
and taking advantage that, in the two-dimensional setting, the stretching term identically vanishes, we find that $\o$ solves
the transport equation
\begin{equation} \label{eq:omega}
 \d_t\o\,+\,u\cdot\nabla\o\,=\,-\,\alpha\,\Delta m\,.
\end{equation}
Despite the very simple structure of this equation, its right-hand side represents an obstacle when trying to propagate $L^q$ bounds for $\o$, as it involves
two full derivatives of $m$. This does not combine very well with the transport-diffusion structure of the equation of $m$, at least in the endpoint
case $q=+\infty$.

For this reason, similarly to what it is done in \cite{F-FD}, and crucially inspired by \cite{H-K-R}, we introduce the auxiliary unknown
\begin{equation} \label{eq:def_G}
 \Gamma\,:=\,\frac{1}{\alpha}\,\o\,+\,\frac{1}{\k}\,m\,.
\end{equation}
Combining the equation for $\o$ and the one for $m$, we easily see that $\Gamma$ satisfies the following transport equation with forcing:
\begin{equation} \label{eq:Gamma}
 \d_t\Gamma\,+\,u\cdot\nabla\Gamma\,=\,\frac{\alpha}{\k}\,\o\,=\,\frac{\alpha^2}{\k}\,\Gamma\,-\,\frac{\alpha^2}{\k^2}\,m\,.
\end{equation}
Applying standard estimates for transport equation in Lebesgue spaces, owing to the fact that $u$ is divergence-free, we immediately get,
for any $q\in[1,+\infty]$ and for any $t\geq0$, the bound
\begin{align}
 \label{est:Gamma}
\left\|\Gamma(t)\right\|_{L^q}\,\leq\,\left\|\Gamma_0\right\|_{L^q}\,+\,C\int^t_0\Big(\left\|\Gamma(\t)\right\|_{L^q}\,+\,\left\|m(\t)\right\|_{L^q}\Big)\,\dd\t\,,
\end{align}
for a suitable constant $C>0$ depending on $\alpha$ and $\k$. Here, with obvious notation, we have defined $\Gamma_0$ by the same formula in \eqref{eq:def_G}, where we
use the initial data $\o_0$ and $m_0$ instead of $\o$ and $m$, respectively.

Let us now switch our attention to the equation for $m$. 
By \tsl{e.g.} applying a pointwise method (see for instance Section 3.1 of \cite{F-GB}) to follow the dynamics of the points of maximum and minimum of $m$, we can derive
the following $L^\infty$ bounds for this scalar field: for any $t\geq0$, we have
\begin{align}
 \label{est:m-inf}
\left\|m(t)\right\|_{L^\infty}\,\leq\,\left\|m_0\right\|_{L^\infty}\,+\,C\,
\int^t_0\Big(\left\|\Gamma(\t)\right\|_{L^\infty}\,+\,\left\|m(\t)\right\|_{L^\infty}\Big)\,\dd\t\,,
\end{align}
where again the constant $C>0$ depends only on $\alpha$ and $\k$.
Next, for $q<+\infty$, we multiply the equation for $m$ by the quantity $|m|^{q-2}\,m$ and we integrate over $\R^2$: after noticing that
\[
 \forall\,q\in\,]1,+\infty[\,,\qquad \int_{\R^2}(-\Delta m)\,|m|^{q-2}\,m\,\dx\,=\,(q-1)\int_{\R^2}|m|^{q-2}\,|\nabla m|^2\,\dx\,\geq\,0\,,
\]
we obtain the following estimate, for any $q\in\,]1,+\infty[\,$ and for a constant $C=C(\alpha,\k)>0$ depending only on the quantities in the brackets:
\begin{align}
 \label{est:m-q}
\left\|m(t)\right\|_{L^{q}}\,\leq\,\left\|m_0\right\|_{L^{q}}\,+\,C\,\int^t_0\Big(\left\|\Gamma(\t)\right\|_{L^{q}}\,+\,\left\|m(\t)\right\|_{L^{q}}\Big)\,\dd\t\,.
\end{align}

Summing up inequalities \eqref{est:Gamma}, \eqref{est:m-inf} and \eqref{est:m-q} and applying Gr\"onwall lemma, we finally deduce the following
global bounds for the Lebesgue norms of $\Gamma$ and $m$. For the sake of clarity, we divide our discussion depending on whether
$p<+\infty$ (the case of Theorem \ref{th:global-strong}) or $p=+\infty$ (as in Theorem \ref{th:endpoint}).

\paragraph*{The case of finite $p$.}
Assume to be first in the situation of Theorem \ref{th:global-strong}, hence $p\in\,]1,+\infty[\,$.
By embeddings \eqref{eq:BesovLp}, \eqref{eq:AnnLInfty} and \eqref{cond:Lipschitz}, we have that both $\o_0$ and $m_0$ belong to
the space $L^p\cap L^\infty$, hence so does $\Gamma_0$.

As a consequence, we apply estimate \eqref{est:Gamma} with $q=p$ and $q=+\infty$, estimate \eqref{est:m-inf} and estimate \eqref{est:m-q} with $q=p$. We thus get,
for a constant $C=C(\alpha,\k)>0$ as above, the inequality
\begin{align}
\label{est:G-m_final}
 \forall\,t\geq0\,,\qquad \left\|\Gamma(t)\right\|_{L^{p}\cap L^\infty}\,+\,\left\|m(t)\right\|_{L^{p}\cap L^\infty}\,\leq\,
 \Big(\left\|\o_0\right\|_{L^{p}\cap L^\infty}\,+\,\left\|m_0\right\|_{L^{p_0}\cap L^\infty}\Big)\,e^{C\,t}\,,
\end{align}
where we have set $\|f\|_{X\cap Y}\,:=\,\|f\|_X\,+\,\|f\|_Y$ and where we have also used the expression of $\Gamma_0$ in terms of $\o_0$ and $m_0$.

\begin{rmk} \label{r:m-L^2}
In the special case $p=2$, the above estimate holds true with, in addition, the quantity
$\|\nabla m\|_{L^2_t(L^2)}$ on its left-hand side.
\end{rmk}

Of course, recovering $\o$ in terms of $\Gamma$ and $m$ from relation \eqref{eq:def_G} yields a similar bound also for the vorticity:
\begin{align}
\label{est:o_final}
 \forall\,t\geq0\,,\qquad \left\|\o(t)\right\|_{L^{p}\cap L^\infty}\,\leq\,
 \Big(\left\|\o_0\right\|_{L^{p}\cap L^\infty}\,+\,\left\|m_0\right\|_{L^{p}\cap L^\infty}\Big)\,e^{C\,t}\,,
\end{align}
where $C=C(\alpha,\k)>0$ only depends on the quantities inside the brackets, as before.


\paragraph*{The case $p=+\infty$.}
In the framework of Theorem \ref{th:endpoint}, instead, from the Besov regularity assumption we only deduce that 
$\o_0$ and $m_0$ belong to $L^\infty$.
However, here we make use of the additional assumption that both $\o_0$ and $m_0$ belong to $L^{p_0}(\R^2)$, with $p_0\in\,]1,2[\,$.
This happens also under assumption (G.3) of Theorem \ref{th:global-strong}.

Therefore, we get estimates analogous to \eqref{est:G-m_final} and \eqref{est:o_final} for the $L^{p_0}\cap L^\infty$ norm of the solution:
\begin{align}
\label{est:G-m-o_inf}
\forall\,t\geq0\,,\qquad 
\left\|m(t)\right\|_{L^{p_0}\cap L^\infty}\,+\,\left\|\o(t)\right\|_{L^{p_0}\cap L^\infty}\,\leq\,
 \Big(\left\|\o_0\right\|_{L^{p_0}\cap L^\infty}\,+\,\left\|m_0\right\|_{L^{p_0}\cap L^\infty}\Big)\,e^{C\,t}\,,
\end{align}
for a constant $C=C(\alpha,\k)>0$ depending on the quantities inside the brackets only.

\begin{rmk} \label{r:Leb}
It is worth noticing that estimates \eqref{est:G-m_final} and \eqref{est:o_final}, as well as \eqref{est:G-m-o_inf}, are \emph{global} in time.
\end{rmk}

\subsubsection{Bounds for the velocity field} \label{sss:Leb_u}

To conclude this part, we exhibit bounds also for the Lebesgue norm of the velocity field $u$.
This will be particularly useful in Subsection \ref{ss:bes} below, where we need to propagate the Besov regularity norms of the solution.

Let us first consider the case $1<p<+\infty$, corresponding to the statement of Theorem \ref{th:global-strong}.
Performing a classical $L^p$ estimate for the momentum equation, namely the equation for $u$, we get, for a suitable universal constant $C>0$ and
for any time $t\geq0$, the bound
\begin{equation} \label{est:u-L^p}
 \left\|u(t)\right\|_{L^p}\,\leq\,\left\|u_0\right\|_{L^p}\,+\,\int^t_0\Big(\left\|\nabla\Pi(\t)\right\|_{L^p}\,+\,\left\|\nabla m(\t)\right\|_{L^p}\Big)\,\dd\t\,.
\end{equation}

In order to bound the right-hand side of \eqref{est:u-L^p}
in the case $p<+\infty$, we compute the divergence of the momentum equation. We thus discover
that the pressure solves the elliptic problem
\[
 -\,\Delta\Pi\,=\,\div\Big((u\cdot\nabla)u\Big)\,,
\]
which implies that
\begin{equation} \label{eq:pressure}
 \nabla \Pi\,=\,\nabla(-\Delta)^{-1}\div\Big((u\cdot\nabla)u\Big)\,,
\end{equation}
where the equality should be read in the sense of Fourier multipliers. Thus, since $p\in\,]1,+\infty[\,$, Calder\'on-Zygmund theory implies that
\begin{equation} \label{est:pressure}
 \left\|\nabla\Pi\right\|_{L^p}\,\lesssim\,\left\|u\right\|_{L^p}\,\left\|\nabla u\right\|_{L^\infty}\,,
\end{equation}
where the implicit multiplicative constant depends on $p$.

By exploiting this estimate, we find that
\begin{align*}
\forall\,t\geq0\,,\qquad
\left\|u(t)\right\|_{L^p}\,\lesssim\,\left\|u_0\right\|_{L^p}\,+\,\int^t_0\Big(\left\|u\right\|_{L^p}\,\left\|\nabla u\right\|_{L^\infty}\,+\,
\left\|\nabla m(\t)\right\|_{L^p}\Big)\,\dd\t\,.
\end{align*}

We now have to bound the last term appearing in the previous estimate. Since $s-1\geq d/p>0$, by embddings and interpolation in Besov spaces (see, e.g., \cite[Theorem 2.80]{BCD}), we can write
\begin{equation}\label{EmbeddingBCD}
 \left\|\nabla m\right\|_{L^p}\,\leq\,\left\| m\right\|_{B^s_{p,1}}\,\leq\,
\left\| m\right\|_{B^{s-1}_{p,1}}^{1/2}\,\left\| m\right\|_{B^{s+1}_{p,1}}^{1/2}\,.
\end{equation}
Applying the Young inequality, we thus infer that
\begin{align}
\label{est:u-Leb_p}
\forall\,t\geq0\,,\quad
\left\|u(t)\right\|_{L^p}\,&\leq\,\left\|u_0\right\|_{L^p}\,+\,C(\delta)\int^t_0\Big(\left\|u\right\|_{L^p}\,\left\|\nabla u\right\|_{L^\infty}\,+\,
\left\| m(\t)\right\|_{B^{s-1}_{p,1}}\Big)\,\dd\t \\
\nonumber
&\qquad\qquad\qquad\qquad\qquad\qquad\qquad\qquad\qquad
\,+\,\delta\int^t_0\left\| m(\t)\right\|_{B^{s+1}_{p,1}}\,\dd\t\,,
\end{align}
where the inequality holds true for any $\delta\in\,]0,1[\,$, whose value will be fixed later.

\medbreak
Of course, estimate \eqref{est:u-Leb_p} can be improved in the case $p=2$, or whenever we dispose of the assumption that
$u_0$ and $m_0$ additionally belong to $L^2(\R^2)$. Observe that assumptions (G.1) and (G.2) of Theorem \ref{th:global-strong} fall into this framework.

Indeed, performing an $L^2$ estimate for the momentum equation, in \eqref{est:u-L^p} we can get rid of the pressure term, owing to the fact that
$\nabla\Pi$ is orthogonal to $u$ in $L^2$, since $\div u=0$. More precisely, straightforward computations yield, for any $\delta\in\,]0,1[\,$,
the estimate
\[
 \left\|u(t)\right\|^2_{L^2}\,\leq\,\left\|u_0\right\|_{L^2}^2\,+\,C(\delta)\int^t_0\|u(\t)\|_{L^2}^2\,\dd\t\,+\,
\delta\int^t_0\left\|\nabla m(\t)\right\|^2_{L^2}\,\dd\t\,.
\]
At the same time, we perform an energy estimate in the equation for $m$: differently from the computations leading to \eqref{est:m-q}, we integrate
by parts the vorticity term and we apply the Young inequality: we find,
for any $\delta\in\,]0,1[\,$ to be chosen later, the inequality
\[
\left\|m(t)\right\|^2_{L^{2}}\,+\,\int^t_0\left\|\nabla m(\t)\right\|^2_{L^2}\,\dd\t\,\leq\,
\left\|m_0\right\|^2_{L^{2}}\,+\,C(\delta)\int^t_0\|u(\t)\|_{L^2}^2\,\dd\t\,+\,
\delta\int^t_0\left\|\nabla m(\t)\right\|^2_{L^2}\,\dd\t\,.
\]
Combining these bounds and choosing $\de>0$ small enough, it is easy to get a global bound for the $L^2$ norms of $u$ and $m$:
\begin{align}
\label{est:u-m-energy}
\forall\,t\geq0\,,\qquad \left\|u(t)\right\|^2_{L^2}\,+\,\left\|m(t)\right\|^2_{L^2}\,+\int^t_0\left\|\nabla m(\t)\right\|^2_{L^2}\,\dd\t\,\lesssim\,
\Big(\left\|u_0\right\|^2_{L^{2}}\,+\,\left\|m_0\right\|^2_{L^2}\Big)\,e^{C\,t}\,,
\end{align}
for a universal constant $C=C(\alpha,\k)>0$ depending only on $\alpha$ and $\k$.

We observe that using the previous estimate would yield an analogous version of Theorem \ref{th:endpoint} under a finite energy condition over $u_0$
and $m_0$, as claimed in Remark \ref{r:p=inf}.

\medbreak
In the case $p=+\infty$, as in Theorem \ref{th:endpoint}, we cannot argue anymore in the previous way. However, we can make use of the additional
integrability assumption over $\o_0$ and $m_0$, combined with the following well-known result (see \tsl{e.g.} \cite{Maj-Bert} and \cite{F-FD} for a proof).
\begin{prop} \label{p:BS}
Let $p_0\,\in\,\,]1,2[\,$ and define $q_0\,\in\,\,]2,+\infty[$ as
\[
\frac1{q_0}\,:=\,\frac1{p_0}\,-\,\frac12\,.
\]
Assume that the vorticity $\omega$ satisfies
$\omega\in L^{p_0}(\mathbb R^2)\cap L^\infty(\mathbb R^2)$,
and let $u$ be the associated velocity field given by the Biot-Savart law
\begin{equation} \label{eq:BS}
u\,=\,-\,\nabla^\perp(-\Delta)^{-1}\o\,=\,\frac{1}{2\pi}\,\int_{\R^2}\dfrac{1}{|x-y|^2}\,(x-y)^\perp\,\o(y)\,\dd y\,.
\end{equation}

Then, $u\in L^{q_0}(\R^2)\cap L^\infty(\R^2)$. Moreover, there exists a universal constant $C>0$ such that
\[
\|u\|_{L^{q_0}}\,\leq\,C\,\|\omega\|_{L^{p_0}} \qquad \mbox{ and }\qquad
\|u\|_{L^\infty}\,\leq\,C\,\big(\|\omega\|_{L^{p_0}}\,+\,\|\o\|_{L^\infty}\big)\,.
\]
\end{prop}

Therefore, by applying Proposition \ref{p:BS} and using estimate \eqref{est:G-m-o_inf}, when $p=+\infty$ we deduce that 
\begin{align}
\label{est:u-Leb_inf}
\forall\,t\geq0\,,\quad
\left\|u(t)\right\|_{L^{q_0}\cap L^\infty}\,&\leq\,\Big(\left\|\o_0\right\|_{L^{p_0}\cap L^\infty}\,+\,\left\|m_0\right\|_{L^{p_0}\cap L^\infty}\Big)\,e^{C\,t}\,.
\end{align}

\begin{rmk} \label{r:u-Leb_p}
The same estimate \eqref{est:u-Leb_inf} holds true also in the case
\[
 1\,<\,p\,<\,2\,,
\]
by simply replacing $p_0$ by $p$ (notice that, by Besov embeddings, in that case $\o_0$ and $m_0$ both belong to $L^{p}\cap L^\infty$)
and $q_0$ by $q$, defined in a similar way \tsl{via} the formula $1/q\,:=\,1/p\,-\,1/2$.
\end{rmk}

\subsection{Estimates for the Besov regularity}
\label{ss:bes}

We now establish \tsl{a priori} estimates for the Besov norms of $u$ and $m$.
Recall that $u_0$ and $m_0$ are assumed to belong, respectively, to the spaces $B^s_{p,1}$ and $B^{s-1}_{p,1}$.
We will treat the two cases $p<+\infty$ and $p=+\infty$ together, with however some necessary specifications when $s=1$ (since we recall that $B^0_{p,1}$
is not an algebra).

We start by noticing that, after cutting into low and high frequencies and using the first Bernstein inequality of Lemma \ref{l:bern}, we can estimate
\begin{equation} \label{est:low-high}
 \left\|u\right\|_{B^s_{p,1}}\,\lesssim\,\left\|u\right\|_{L^p}\,+\,\left\|\o\right\|_{B^{s-1}_{p,1}}\,,
\end{equation}
where $p$ can either satisfy $1<p<+\infty$ (as in the framework of Theorem \ref{th:global-strong})
or be $p=+\infty$ (as in the setting of Theorem \ref{th:endpoint}).

As a consequence, we can perform estimates for the Besov norms of the solutions by working with $\o$ and $m$, both at $B^{s-1}_{p,1}$  level of regularity.
As a matter of fact, for the same reasons mentioned in Subsection \ref{ss:leb}, it is more convenient to work with the unknown
$\Gamma$ defined in \eqref{eq:def_G}, rather than working with $\o$.

Before proceeding, we point out that, in order to bound the $B^s_{p,1}$ norm of the solution, we could directly use Theorem \ref{th:transport}
for equation \eqref{eq:Gamma} and Theorem \ref{t:tr-diff} for the second equation in \eqref{eq:2D-microp}. However, we prefer here to
do all the estimates ``by hands'': in this way, we will obtain much more precise estimates, which lay the ground to the proof of the global well-posedness
(performed in Section \ref{s:existence} below).

\medbreak
In order to get the sought estimates, we are going to localise the equations for $m$ and $\Gamma$ in frequencies. For convenience, in the computations below
we adopt the following notation: for any $j\geq-1$, we set
\[
 f_j\,=\,\Delta_jf\,,\qquad\qquad \mbox{ for }\qquad f\,\in\,\big\{u,\o,m,\Gamma\big\}\,.
\]
In addition, for simplicity of presentation, hereafter we set $\alpha=\k=1$.

Applying the operator $\Delta_j$, for $j\geq-1$, to the equations for $\Gamma$ and $m$, we obtain
\begin{align*}
\d_t \Gamma_j\,+\,u\cdot\nabla \Gamma_j\,&=\,[u\cdot\nabla,\Delta_j]\Gamma\,+\,\Gamma_j\,-\,m_j\,, \\
\d_t m_j\,-\, \Delta m_j\,&=\,-\,\Delta_j\left( u\cdot\nabla m\right) +\,\Gamma_j\,-\,m_j\,. 
\end{align*}
Notice that both $\Gamma_j$ and $m_j$ are real-valued, because $\Gamma$ and $m$ are, and the functions involved in the dyadic partition of unity are radial.

We start by considering the equation for $\Gamma_j$.
We perform $L^p$ estimates on it, for $1<p\leq+\infty$, as we did to prove \eqref{est:Gamma} in Subsection \ref{ss:leb} above: for any $j\geq-1$
and any $t\geq0$, we find
\[
 \left\|\Gamma_j(t)\right\|_{L^p}\,\leq\,\left\|\Delta_j\Gamma_0\right\|_{L^p}\,+\,\int^t_0\Big(\left\|[u\cdot\nabla,\Delta_j]\Gamma\right\|_{L^p}\,+\,
 \left\|\Gamma_j\right\|_{L^p}\,+\,\left\|m_j\right\|_{L^p}\Big)\,\dd\t\,.
\]
Now, we multiply both sides of the previous relation by $2^{j(s-1)}$ and we perform a sum over $j\geq-1$: we obtain, for all $t\geq0$, the bound
\begin{align*} 
\left\|\Gamma(t)\right\|_{B^{s-1}_{p,1}}\,&\leq\,\left\|\Gamma_0\right\|_{B^{s-1}_{p,1}} \\
\nonumber
&\qquad \,+\,
\int^t_0\left(\sum_{j\geq-1}2^{j(s-1)}\,\big\|[u\cdot\nabla,\Delta_j]\Gamma\big\|_{L^p}\,+\,\left\|\Gamma\right\|_{B^{s-1}_{p,1}}\,+\,
\left\|m(t)\right\|_{B^{s-1}_{p,1}}\right)\,\dd\t\,.
\end{align*}
According to the second estimate in Lemma \ref{l:CommBCD}, we can bound
\begin{align}\label{est:G_Bes0}
 \sum_{j\geq-1}2^{j(s-1)}\,\big\|[u\cdot\nabla,\Delta_j]\Gamma\big\|_{L^p}\,&\lesssim\,\|\nabla u\|_{L^\infty}\,\|\Gamma\|_{B^{s-1}_{p,1}}\,+\,
\|\o\|_{B^{s-1}_{p,1}}\,\|\Gamma\|_{L^\infty} \\
&\lesssim\,\Big(\left\|\nabla u\right\|_{L^\infty}\,+\,\|\Gamma\|_{L^\infty}\Big)\,\|\Gamma\|_{B^{s-1}_{p,1}}\,+\,\|\Gamma\|_{L^\infty}\,\|m\|_{B^{s-1}_{p,1}}\,. \nonumber 
\end{align}
Let us introduce the energy functional
\begin{equation} \label{eq:def_E}
 E^s_p(t)\,:=\,\left\|u\right\|_{L^\infty_t(B^{s}_{p,1})}\,+\,\left\|m\right\|_{L^\infty_t(B^{s-1}_{p,1})\cap L^1_t(B^{s+1}_{p,1})}\,,
\end{equation}
where we agree that $E^s_p(0)$ denotes the same quantity, defined by using the initial datum $\big(u_0,m_0\big)$ instead of the functions $\big(u(t),m(t)\big)$.
Observe that, owing to \eqref{est:low-high}, for any $t\geq0$ one has
\[
 E^s_p(t)\,\approx\,
 \left\|u\right\|_{L^\infty_t(L^p)}\,+\,\left\|\o\right\|_{L^\infty_t(B^{s-1}_{p,1})}\,+\,\left\|m\right\|_{L^\infty_t(B^{s-1}_{p,1})\cap L^1_t(B^{s+1}_{p,1})}\,.
\]
Recall that $A\approx B$ means that there exists a universal constant $C>0$
(independent of the data, of the solution and of the time $t\geq 0$) such that $C^{-1}B\leq A\leq C B$.
In particular, keeping definition \eqref{eq:def_G} in mind, we have that
\[
 \left\|\Gamma(t)\right\|_{B^{s-1}_{p,1}}\,+\,\left\|m(t)\right\|_{B^{s-1}_{p,1}}\,\lesssim\,E^s_p(t)\,.
\]
With this notation at hand, 
in light of \eqref{est:Gamma} when $p<+\infty$, or $\eqref{est:G-m-o_inf}$ for $p=+\infty$, from \eqref{est:G_Bes0} we can bound
\begin{align} \label{est:G_Bes}
\left\|\Gamma(t)\right\|_{B^{s-1}_{p,1}}\,&\leq\,\left\|\o_0\right\|_{B^{s-1}_{p,1}}\,+\,\left\|m_0\right\|_{B^{s-1}_{p,1}} \\
\nonumber
&\qquad\qquad \,+\,
C\,\int^t_0\Big(1+\left\|\nabla u(\t)\right\|_{L^\infty}+\|\o(\t)\|_{L^\infty}+\|m(\t)\|_{L^\infty}\Big)\,E^s_p(\t)\,\dd\t\,.
\end{align}

Let us now switch our attention to the equation for $m$. By using Theorem \ref{t:tr-diff}, for any $1<p\leq +\infty$, the embedding
$\wtilde L^\infty_T(B^\s_{p,r})\,\hookrightarrow\,L^\infty_T(B^\s_{p,r})$ and the equality $\wtilde L^1_T(B^\s_{p,1}) \,\equiv \,
L^1_T(B^\s_{p,1})$, we immediately get, for all times $t\geq0$, the estimate
\begin{align*}
&\left\|m\right\|_{L^\infty_t(B^{s-1}_{p,1})}\,+\,\int^t_0\left\|m(\t)\right\|_{B^{s+1}_{p,1}} \\
&\qquad\qquad \lesssim\,\big(1+t\big)\, \left(\left\|m_0\right\|_{B^{s-1}_{p,1}}\,+\,
\int^t_0\left(\big\|u\cdot\nabla m\big\|_{B^{s-1}_{p,1}}\,+\,\left\|\Gamma\right\|_{B^{s-1}_{p,1}}\,+\,\left\|m(t)\right\|_{B^{s-1}_{p,1}}\right)\,\dd\t\right)\,.
\end{align*}
Next, we treat the term $u\cdot\nabla m$ appearing on the right-hand side of the previous inequality. For this, we use the Bony paraproduct decomposition
\eqref{eq:bony}. In fact, in order to have a unified approach which works also for the case $s=1$ (as it may happen in the framework of Theorem \ref{th:endpoint}),
it is more convenient to use the divergence-free condition over $u$ and write
\[
 u\cdot\nabla m\,=\,\mc T_u\nabla m\,+\,\mc T_{\nabla m}u\,+\,\sum_{j=1,2}\d_j\mc R\big(u^j,m\big)\,.
\]
By using the continuity properties of $\mc T$ and $\mc R$ stated in Proposition \ref{p:op}, we can bound
\begin{align*}
 \left\|\mc T_u\nabla m\right\|_{B^{s-1}_{p,1}}\,&\lesssim\,\|u\|_{L^\infty}\,\left\|\nabla m\right\|_{B^{s-1}_{p,1}}\,\lesssim\,
\|u\|_{L^\infty}\,\left\|m\right\|_{B^{s}_{p,1}}, \\
\left\|\d_j\mc R\big(u^j,m\big)\right\|_{B^{s-1}_{p,1}}\,&\lesssim\,\left\|\mc R\big(u^j,m\big)\right\|_{B^{s}_{p,1}}\,\lesssim\,
\|u\|_{B^s_{p,1}}\,\|m\|_{L^\infty}, \\
\left\|\mc T_{\nabla m}u\right\|_{B^{s-1}_{p,1}}\,&\lesssim\,\|\nabla m\|_{B^{-1}_{\infty,\infty}}\,\left\|\nabla u\right\|_{B^{s-1}_{p,1}}\,\lesssim\,
\|m\|_{L^\infty}\,\left\|u\right\|_{B^{s}_{p,1}}\,.
\end{align*}
Hence, from the above bound we deduce, for a suitable universal constant $C>0$, the inequality
\begin{align*}
&\left\|m\right\|_{L^\infty_t(B^{s-1}_{p,1})}\,+\,\int^t_0\left\|m(\t)\right\|_{B^{s+1}_{p,1}} \\
&\quad \leq\,C\,\big(1+t\big)\, \left(\left\|m_0\right\|_{B^{s-1}_{p,1}}\,+\,
\int^t_0\Big(1+\|m\|_{L^\infty}\Big)\,E^s_p\,\dd\t\right)\,+\,
C\,(1+t)\,\int^t_0\|u\|_{L^\infty}\,\left\|m\right\|_{B^{s}_{p,1}}\,\dd\t\,.
\end{align*}
We now focus on the last term appearing on the right-hand side. For it, we can use the interpolation inequality \eqref{EmbeddingBCD} together with Young inequality to write, for any $t\geq0$ fixed and for any $\delta\in\,]0,1[\,$, the bound
\begin{align*}
\|u\|_{L^\infty}\,\left\|m\right\|_{B^{s}_{p,1}}\,\lesssim\,C(\delta)\,(1+t)\,\|u\|_{L^\infty}^2\,\left\|m\right\|_{B^{s-1}_{p,1}}\,+\,
\frac{\delta}{1+t}\,\left\|m\right\|_{B^{s+1}_{p,1}}\,.
\end{align*}
Choosing $\delta>0$ small enough, we can thus absorb the higher order term into the left-hand side of the previous inequality. This finally yields,
for any $t\geq0$, the estimate
\begin{align}
\label{est:m-Bes}
&\left\|m\right\|_{L^\infty_t(B^{s-1}_{p,1})}\,+\,\int^t_0\left\|m(\t)\right\|_{B^{s+1}_{p,1}} \\
\nonumber
&\qquad\qquad \lesssim\,\big(1+t\big)^2\, \left(\left\|m_0\right\|_{B^{s-1}_{p,1}}\,+\,
\int^t_0\Big(1\,+\,\|u(\t)\|_{L^\infty}^2\,+\,\|m(\t)\|_{L^\infty}\Big)\,E^s_p(\t)\,\dd\t\right)\,.
\end{align}
In the end, we can combine the above estimates to get a convenient bound for the Besov norm of the solution $\big(u,m\big)$.
For this, let us divide our discussion into the two cases $1<p<+\infty$ (Theorem \ref{th:global-strong}) and $p=+\infty$ (Theorem \ref{th:endpoint}).

\paragraph*{The case $1<p<+\infty$.}
When $p<+\infty$ is finite, we put estimates
\eqref{est:low-high}, \eqref{est:u-Leb_p}, \eqref{est:G_Bes} and \eqref{est:m-Bes} all together.
We thus find that
\begin{align}
\label{est:Energy_p}
\forall\,t\geq0\,,\qquad E^s_p(t)\,&\lesssim\,\big(1+t\big)^2\,\left(E^s_p(0)\,+\,
\int^t_0\mc I(\t)\,E^s_p(\t)\,\dd\t\right)\,,
\end{align}
where we have defined the function
\begin{equation} \label{eq:def_I}
 \mc I(\t)\,:=\,1+\|\nabla u(\t)\|_{L^\infty}+\|u(\t)\|_{L^\infty}^2+\|\o(\t)\|_{L^\infty}+\|m(\t)\|_{L^\infty}\,.
\end{equation}

\paragraph*{The case $p=+\infty$.}
When $p=+\infty$, instead, we have to replace estimate \eqref{est:u-Leb_p} by \eqref{est:u-Leb_inf}. By combining it again with \eqref{est:low-high},
\eqref{est:G_Bes} and \eqref{est:m-Bes}, we obtain
\begin{align}
\label{est:Energy_inf}
\forall\,t\geq0\,,\qquad E^s_p(t)\,&\lesssim\,e^{C\,t}\,\left(E^s_p(0)\,+\,\left\|\o_0\right\|_{L^{p_0}}\,+\,\left\|m_0\right\|_{L^{p_0}}\,+\,
\int^t_0\mc I(\t)\,E^s_p(\t)\,\dd\t\right)\,,
\end{align}
where the function $\mc I$ is the same as above, defined in formula \eqref{eq:def_I}.

\paragraph*{}
With estimates \eqref{est:Energy_p} and \eqref{est:Energy_inf} at hand, it is a routine matter to establish uniform estimates for the Besov norms of
$u$ and $m$ in some finite time interval $[0,T]$. Let us present the argument in detail below.

First of all, we can use the embeddings $B^{s-1}_{p,1}(\R^2)\hookrightarrow L^\infty(\R^2)$ and $B^s_{p,1}(\R^2)\hookrightarrow W^{1,\infty}(\R^2)$, which
hold true under conditions \eqref{eq:AnnLInfty} and \eqref{cond:Lipschitz} respectively. Notice that these conditions are satisfied under the assumptions
of Theorems \ref{th:global-strong} and \ref{th:endpoint}. We can thus bound
\[
 \mc I(\t)\,\lesssim\,1\,+\,\big(E^s_p(\t)\big)^2\,.
\]
Inserting this inequality into \eqref{est:Energy_p} when $p<+\infty$, or into \eqref{est:Energy_inf} when $p=+\infty$, in turn we find
\begin{align} \label{est:Energy-tot}
\forall\,t\geq0\,,\qquad E^s_p(t)\,\lesssim\,f_p(t)\,\left(\mc N_p(0)\,+\,\int^t_0\Big(1\,+\,\big(E^s_p(\t)\big)^2\Big)\,E^s_p(\t)\,\dd\t\right)\,,
\end{align}
where, for convenience of notation, we have set 
\begin{align*}
& f_p(t)\,=\,(1+t)^2\quad \mbox{ and }\quad \mc N_p(0)\,=\,E^s_p(0)\qquad\qquad\qquad\qquad\qquad\qquad \mbox{ for }\qquad p\in\,]1,+\infty[\;, \\
& f_p(t)\,=\,e^{C\,t}\quad \mbox{ and }\quad \mc N_p(0)\,=\,E^s_p(0)\,+\,\left\|\o_0\right\|_{L^{p_0}}\,+\,\left\|m_0\right\|_{L^{p_0}}
\qquad\qquad \mbox{ for }\qquad p=+\infty\,.
\end{align*}

Next, let us define the time $T>0$ as
\[
 T\,:=\,\sup\left\{ t>0\quad \Big|\qquad \int^t_0\big(E^s_p(\t)\big)^2\,\dd\t\,\leq\,\log2\right\}\,.
\]
By time continuity of the solution with values in the respective Besov space, we see that $T$ is indeed well-defined and positive.
Without loss of generality, we can assume that $T<+\infty$. As a matter of fact, if $T=+\infty$ the same argument below applies in any time interval
$[0,T_0]$, for any $T_0>0$, and we can conclude in the same manner.

So, assume $T<+\infty$. Then, from \eqref{est:Energy-tot} we deduce that 
\[
\forall\,t\in[0,T]\,,\qquad  E^s_p(t)\,\lesssim\,f_p(T)\,\left(\mc N_p(0)\,+\,\int^t_0\Big(1\,+\,\big(E^s_p(\t)\big)^2\Big)\,E^s_p(\t)\,\dd\t\right)\,,
\]
and an application of the Gr\"onwall lemma implies that
\begin{equation} \label{est:Energy-unif}
 \forall\,t\in[0,T]\,,\qquad E^s_p(t)\,\leq\,C_T\,\mc N_p(0)\,,
\end{equation}
where the multiplicative constant $C_T>0$ does not depend on the solution nor on the initial datum, but it does depend on the time $T$.
More precisely, one has the explicit expression
\[
 C_T\,\approx\,f_p(T)\,\exp\Big(f_p(T)\,\big(T\,+\,\log2\big)\Big)\,,
\]
which could be used to find a first, rough lower bound (implicit, as depending on the function $f_p(t)$, though) for the lifespan of the solution. In fact,
we are going to see that the solution is global, at least under suitable assumptions, so finding this bound is not really relevant for us.
However, the important fact to be pointed out here is that $T$ can be bounded from below by a quantity which only depends on the fixed value of
$p\in\,]1,+\infty]$ and on the functional norms of the initial datum: there exists a decreasing positive function $F:\R_+\longrightarrow\R_+$,
with $F(0^+)=+\infty$ and $F(\z)\longrightarrow0$ for $\z\to+\infty$, such that
\begin{equation} \label{est:T-lower}
 T\,\geq\,F\big(\mc N_p(0)\big)\,.
\end{equation}

We conclude this part by formulating a remark on the regularity of the pressure function.

\begin{rmk} \label{r:pressure}
Similarly to the analysis of Section 3 of \cite{D_2010} (see also \cite{D:F} for the case $p=+\infty$), from an inspection of equation \eqref{eq:pressure}
one could further obtain regularity for the pressure gradient. More precisely, under the uniform bound \eqref{est:Energy-unif} one has
\[
 \nabla\Pi\,\in\,L^\infty\big([0,T];B^s_{p,1}\big)\,.
\]
In fact, since $u$ is continuous in time with values in the Besov space $B^s_{p,1}$, the pressure gradient inherits this property and it is also continuous in time
with values in the same space $B^s_{p,1}$.

However, this regularity property for $\nabla\Pi$ is mostly not needed in our analysis. We will only use it at the end of the existence proof (see \tsl{e.g.}
the end of Paragraph \ref{sss:convergence}), in order to establish the time continuity of the velocity field.
Therefore, for the sake of simplicity, we avoid to perform this analysis here.
\end{rmk}

\section{Proof of uniqueness} \label{s:uniqueness}
In this section, we prove the uniqueness part of Theorems \ref{th:global-strong} and \ref{th:endpoint}. This is a consequence of suitable stability
results, of independent interest, which we are going to state and prove in Subsection \ref{ss:stab}.
How to apply them to get uniqueness is postponed to Subsection \ref{ss:unique}.

In passing, let us mention the following important point. Uniqueness in Theorem \ref{th:endpoint} is actually a consequence of the well-posedness
result from \cite{F-FD}, which states global well-posedness of Yudovich-type solutions to system \eqref{eq:2D-microp}.
The same result applies also under the assumptions of Theorem \ref{th:global-strong}, whenever $p\in\,]1,2[\,$ (indeed, in that case one has
$\o_0,m_0\,\in\;L^{p}(\R^2)\cap L^\infty(\R^2)$, which is what is needed in \cite{F-FD}).
Actually, it would not be hard to adapt that argument to get uniqueness for the case $1<p<+\infty$ in some other special cases,
namely under additional assumptions guaranteeing the finite-energy condition $u\in L^\infty_T(L^2)$ (for instance, this can be reached if $p=2$, or
if $2< p\leq 4$, provided one additionally requires $m_0\in L^2$).

Despite the previous discussion, in this section we will seek general, and quantitative, stability results valid for all the cases $1<p\leq +\infty$.
Indeed, this stability result will be important not only for obtaining uniqueness, but also for the construction of a solution,
when proving strong convergence of a suitably constructed sequence of smooth approximate solutions to a true solution of equations \eqref{eq:2D-microp}.

\subsection{Stability results} \label{ss:stab}

Motivated by the previous discussion, here we state and prove general stability estimates for solutions to system \eqref{eq:2D-microp} at the level
of regularity of Theorems \ref{th:global-strong} and \ref{th:endpoint}. We split our analysis into two main cases:
on the one hand, the general case $p\in\,]1,+\infty[\,$, in which higher regularity for the vorticity is needed, and,
on the other hand, the case $u\in L^2$ (implied by suitable assumptions, either on $p$ or on the initial data), in which an argument similar
to the one used in \cite{F-FD} applies.

Let us start by considering a general case, with no restrictions on $p\in\,]1,+\infty[\,$, which nonetheless requires the rather stringent assumption
$\nabla\o\in L^1_T(L^\infty)$.
The first main result of this part is contained in the following statement.
\begin{prop} \label{p:stab}
 Let $\big(u_1,m_1\big)$ and $\big(u_2,m_2\big)$ be two weak solutions to system \eqref{eq:2D-microp}, defined on a common time interval $[0,T]$, for some $T>0$.
Assume that:
\begin{enumerate}[\rm (i)]
 \item the differences $\delta u\,:=\,u_1\,-\,u_2$ and $\delta m\,:=\,m_1\,-\,m_2$  belong to
 $L^\infty\big([0,T];W^{1,p}(\R^2)\big)\,\cap\,W^{1,\infty}\big([0,T];L^p(\R^2)\big)$, for some finite $p\in\,]1,+\infty[\,$;
\item $u_1$ belongs to $L^\infty\big([0,T]\times\R^2\big)$;
\item both $\nabla u_1$ and $\nabla u_2$ belong to $L^1\big([0,T];L^\infty(\R^2)\big)$;
 \item $\nabla m_2$ belongs to $L^r\big([0,T];L^\infty(\R^2)\big)$, for some $r\in\,]1,+\infty[\,$;
 \item for the same $r\in\,]1,+\infty[\,$ above, one has that $\de m(0)$ belongs to the homogeneous Besov space $\dot B^{2-\frac{2}{r}}_{p,r}$; 
 \item $\nabla\o_2$ belongs to $L^1\big([0,T];L^\infty(\R^2)\big)$.
\end{enumerate}

Then, there exists a small time $T_0\in\,]0,T]$ such that one has the following estimate:
\begin{align*}
 &\left\|\delta u\right\|_{L^\infty_{T_0}(W^{1,p})}\,+\,\left\|\delta m\right\|_{L^\infty_{T_0}(L^p)}\,+\,\left\|\Delta \de m\right\|_{L^1_{T_0}(L^p)} \\
 &\qquad\qquad\qquad\qquad\qquad
\leq\,C\, \Big(\left\|\delta u(0)\right\|_{W^{1,p}}\,+\,\left\|\delta m(0)\right\|_{L^p\cap \dot B^{2-\frac{2}{r}}_{p,r}}\Big)\,e^{C\,\mc J(T_0)}\,,
\end{align*}
where the function $\mc J(t)$ is defined as
\[
 \mc J(t)\,:=\,\int^t_0\Big(1\,+\,\left\|\nabla u_1(\t)\right\|_{L^\infty}\,+\,\left\|\nabla u_2(\t)\right\|_{L^\infty}\,+\,
\left\|\nabla \o_2(\t)\right\|_{L^\infty}\,+\,\left\|\nabla m_2(\t)\right\|_{L^\infty}\Big)\,\dd\t\,.
\]
\end{prop}

\begin{proof}
 We proceed in several steps. To begin with, we write down the equation for $\de u$: after denoting $\nabla\de\Pi\,:=\,\nabla\Pi_1\,-\,\nabla\Pi_2$,
we can write
\begin{equation} \label{eq:de-u}
 \d_t\de u\,+\,(u_1\cdot\nabla)\de u\,+\,(\de u\cdot\nabla)u_2\,+\,\nabla\de\Pi\,=\,-\,\nabla^\perp\de m\,
\end{equation}
where, for convenience, we have adopted the same convention as in Subsection \ref{ss:bes} and set $\alpha=\k=1$.
We now perform a $L^p$ estimate for $\delta u$ and get, for any $t\in[0,T]$, the bound
\[
 \left\|\delta u(t)\right\|_{L^p}\,\lesssim\,\left\|\delta u(0)\right\|_{L^p}\,+\,\int^t_0\Big(\left\|\nabla\delta\Pi\right\|_{L^p}\,+\,
\left\|\nabla\delta m\right\|_{L^p}\,+\,\left\|\nabla u_2\right\|_{L^\infty}\,\left\|\de u\right\|_{L^p}\Big)\,\dd\t\,.
\]

Next, similarly to what done in Subsection \ref{ss:leb}, we derive the following formula:
\[
 \nabla\de\Pi\,=\,\nabla(-\Delta)^{-1}\div\Big((u_1\cdot\nabla)\de u\,+\,(\de u\cdot\nabla)u_2\Big)\,.
\]
Observing that
\begin{align*}
 \div\Big((u_1\cdot\nabla)\de u\Big)\,&=\,\sum_{j,k=1,2}\d_k\Big(u_1^j\,\d_j\de u^k\Big)\,=\,\sum_{j,k=1,2}\d_j\Big(\de u^k\,\d_ku_1^j\Big)\,,
\end{align*}
which holds true thanks to the divergence-free condition satisfied by both $u_1$ and $\de u$, by Calder\'on-Zygmund theory we can bound
\[
 \left\|\nabla\delta\Pi\right\|_{L^p}\,\lesssim\,\Big(\left\|\nabla u_1\right\|_{L^\infty}\,+\,\left\|\nabla u_2\right\|_{L^\infty}\Big)\,
\left\|\de u\right\|_{L^p}\,.
\]
Thus, after using interpolation and the Young inequality for treating the term $\nabla\de m$, we infer that, for any $t\in[0,T]$, there holds
\begin{align}
\label{est:de-u}
 & \left\|\delta u(t)\right\|_{L^p}\,\leq\,C\,\left\|\delta u(0)\right\|_{L^p}\,+\,
C\int^t_0\left\|\delta m(\t)\right\|_{L^p}\,\dd\t\,+\,\veps\int^t_0\left\|\Delta\delta m(\t)\right\|_{L^p}\,\dd\t \\
\nonumber
& \qquad\qquad\qquad\qquad\qquad\qquad
+\,C\int^t_0\Big(\left\|\nabla u_1(\t)\right\|_{L^\infty}\,+\,\left\|\nabla u_2(\t)\right\|_{L^\infty}\Big)\,\left\|\de u(\t)\right\|_{L^p}\,\dd\t\,,
\end{align}
where the parameter $\veps\in\,]0,1[\,$, so far arbitrary, will be chosen small enough later on. Of course, the multiplicative constant $C>0$ appearing
in the previous estimate also depends on such $\veps$.

Next, we write down the equation for $\de m$: we have
\begin{equation} \label{eq:de-m}
\d_t\de m\,+\,u_1\cdot\nabla \de m\,+\,\de u\cdot\nabla m_2\,-\,\Delta\de m\,=\,\de\o\,,
\end{equation}
where $\de\o\,:=\,\o_1\,-\,\o_2\,=\,\curl\de u$ denotes the vorticity associated to the ``difference velocity'' $\de u$.
Similarly to \eqref{est:m-q}, a direct $L^p$ estimate for it yields, for any $t\in[0,T]$, the bound
\begin{align}
\label{est:de-m}
\left\|\delta m(t)\right\|_{L^p}\,\lesssim\,\left\|\delta m(0)\right\|_{L^p}\,+\,\int^t_0\Big(\left\|\de\o(\t)\right\|_{L^p}\,+\,
\left\|\nabla m_2(\t)\right\|_{L^\infty}\,\left\|\de u(\t)\right\|_{L^p}\Big)\,\dd\t\,.
\end{align}
On the other hand, performing a maximal regularity estimate for it (see \tsl{e.g.} Lemma 7.3 of \cite{LR}), we discover that,
for $r\in\,]1,+\infty[\,$ as in assumption (iv) of the statement and for any $t\in[0,T]$, one has
\begin{align*}
\left\|\Delta \de m\right\|_{L^r_t(L^p)}\,\lesssim\,\left\|\de m(0)\right\|_{\dot B^{2-\frac{2}{r}}_{p,r}}\,+\,
\left\|\de\o\right\|_{L^r_t(L^p)}\,+\,\left\|u_1\cdot\nabla\de m\right\|_{L^r_t(L^p)}\,+\,
\left\|\de u\cdot\nabla m_2\right\|_{L^r_t(L^p)}\,.
\end{align*}
where the homogeneous Besov norm of the initial datum arises as a consequence of Theorem 2.34 of \cite{BCD}.
By using the following bounds,
\begin{align*}
 \left\|\de u\cdot\nabla m_2\right\|_{L^r_t(L^p)}\,&\lesssim\,\left\|\de u\right\|_{L^\infty_t(L^p)}\,\left\|\nabla m_2\right\|_{L^r_t(L^\infty)}\,, \\
 \left\|u_1\cdot\nabla\de m\right\|_{L^r_t(L^p)}\,&\lesssim\,\left\|u_1\right\|_{L^\infty_t(L^\infty)}\,\left\|\nabla\de m\right\|_{L^r_t(L^p)} \\
&\leq\,C_\theta\,\left\|u_1\right\|_{L^\infty_t(L^\infty)}^2\,\left\|\de m\right\|_{L^r_t(L^p)}\,+\,\theta\,\left\|\Delta\de m\right\|_{L^r_t(L^p)}\,,
\end{align*}
where $\theta>0$ can be chosen arbitrarily small, we finally arrive to the inequality
\begin{align}
\label{est:Delta-de-m}
\left\|\Delta \de m\right\|_{L^r_t(L^p)}\,&\lesssim\,\left\|\de m(0)\right\|_{\dot B^{2-\frac{2}{r}}_{p,r}}\,+\,
\left\|\de\o\right\|_{L^r_t(L^p)} \\
\nonumber
&\qquad\qquad +\,\left\|\de u\right\|_{L^\infty_t(L^p)}\,\left\|\nabla m_2\right\|_{L^r_t(L^\infty)}\,+\,
\left\|u_1\right\|_{L^\infty_t(L^\infty)}^2\,\left\|\de m\right\|_{L^r_t(L^p)}\,.
\end{align}

Observe that
\[
 \left\|\Delta \de m\right\|_{L^1_t(L^p)}\,\leq\,t^{1-\frac{1}{r}}\,\left\|\Delta \de m\right\|_{L^r_t(L^p)}\,.
\]
Hence, summing up inequalities \eqref{est:de-u}, \eqref{est:de-m} and \eqref{est:Delta-de-m}, taking $\veps>0$ small enough and making
use of assumptions (ii) and (iv) from the statement, we arrive at
\begin{align}
\label{est:de-u-m}
&\left\|\delta u\right\|_{L^\infty_t(L^p)}\,+\,\left\|\delta m\right\|_{L^\infty_t(L^p)}\,+\,\left\|\Delta \de m\right\|_{L^1_t(L^p)} \\
\nonumber
&\quad \lesssim\,
\left\|\delta u(0)\right\|_{L^p}\,+\,\left\|\delta m(0)\right\|_{L^p\cap \dot B^{2-\frac{2}{r}}_{p,r}}\,+\,
\int^t_0\left\|\delta m(\t)\right\|_{L^p}\,\dd\t\,+\,\int^t_0\left\|\delta \o(\t)\right\|_{L^p}\,\dd\t \\
\nonumber
&\quad\qquad\qquad
+\,\int^t_0\Big(\left\|\nabla u_1(\t)\right\|_{L^\infty}+\left\|\nabla u_2(\t)\right\|_{L^\infty}+\left\|\nabla m_2(\t)\right\|_{L^\infty}\Big)\,
\left\|\de u(\t)\right\|_{L^p}\,\dd\t \\
\nonumber
&\quad\qquad\qquad\qquad\qquad\qquad\qquad\qquad
+\,t^{1-\frac{1}{r}}\,\Big(\left\|\de u\right\|_{L^\infty_t(L^p)}\,+\,\left\|\de\o\right\|_{L^r_t(L^p)}\,+\,\left\|\de m\right\|_{L^r_t(L^p)}\Big)\,.
\end{align}
It is then clear that we can close the stability estimates on some small time interval $[0,T_0]$, provided we establish a suitable bound for the term $\de\o$:
this is our next goal.

Once again, for controlling the difference of the vorticities, we will pass through the auxiliary unknown $\de\Gamma\,:=\,\Gamma_1\,-\,\Gamma_2$, where each
$\Gamma_j$ is defined according to \eqref{eq:def_G}. The equation for $\de\Gamma$ reads
\[
 \d_t\de \Gamma\,+\,u_1\cdot\nabla\de\Gamma\,+\,\de u\cdot\nabla\Gamma_2\,=\,\de\o\,.
\]
Then, a $L^p$ estimate yields the analogue of estimate \eqref{est:de-m}:
\begin{align*}
\left\|\delta \Gamma(t)\right\|_{L^p}\,\lesssim\,\left\|\delta \Gamma(0)\right\|_{L^p}\,+\,\int^t_0\Big(\left\|\de\o(\t)\right\|_{L^p}\,+\,
\left\|\nabla \Gamma_2(\t)\right\|_{L^\infty}\,\left\|\de u(\t)\right\|_{L^p}\Big)\,\dd\t\,,
\end{align*}
for any time $t\in[0,T]$. Thus, combining the previous inequality with \eqref{est:de-m}, we finally find
\begin{align}
\label{est:de-o}
\left\|\delta \o(t)\right\|_{L^p}\,&\lesssim\,\left\|\delta \o(0)\right\|_{L^p}\,+\,\left\|\delta m(0)\right\|_{L^p}\,+\,
\int^t_0\left\|\de\o(\t)\right\|_{L^p}\,\dd\t \\
\nonumber
&\qquad\qquad\qquad +\,
\int^t_0\Big(\left\|\nabla \o_2(\t)\right\|_{L^\infty}+\left\|\nabla m_2(\t)\right\|_{L^\infty}\Big)\,\left\|\de u(\t)\right\|_{L^p}\,\dd\t\,.
\end{align}
At this point, summing up estimates \eqref{est:de-u-m} and \eqref{est:de-o} yields the sought bound.
\end{proof}

Before going on, let us formulate a remark on the assumption $\de m(0)\in\dot B^{2-\frac{2}{r}}_{p,r}$.
\begin{rmk} \label{r:hom-B}
Since $\de m(0)\in W^{1,p}(\R^2)$ by assumption (i), there is no problem in the control of the low frequencies of $\de m(0)$ and one actually gets that
assumption (v) is implied by the non-homogeneous condition $\de m(0)\in B^{s-\frac{2}{r}}_{p,r}$.

Next, observe that, for solutions as in Theorem \ref{th:global-strong}, by interpolation one gets $m\in L^2_T(B^{s}_{p,1})$, which implies by embeddings that
assumption (iv) holds true with $r=2$. Thus, verifying assumption (v) requires to prove $\de m(0)\in B^{1}_{p,2}$.

Now, owing to the stringent assumption (vi), we will be able to apply this stability estimate only when $s\geq 2+2/p$, for which $s-1\geq 1+2/p$.
Thus, in our cases of interest, we will get $m_0\in B^1_{p,2}$ for free simply by embeddings. In other words, assumption (v) is compatible with
the statement of Theorem \ref{th:global-strong}.
\end{rmk}

We also state the counterpart of the above Proposition \ref{p:stab} in the case when we perform a stability estimate in $L^2$.
This will be useful to get uniqueness under less stringent regularity assumptions (that is, lower values of $s$), for some special cases.
We refer to Subsection \ref{ss:unique} below for more details.

\begin{prop} \label{p:stab-energy}
 Let $\big(u_1,m_1\big)$ and $\big(u_2,m_2\big)$ be two weak solutions to system \eqref{eq:2D-microp}, defined on a common time interval $[0,T]$, for some $T>0$.
Assume that:
\begin{enumerate}[\rm (i)]
 \item the differences $\delta u\,:=\,u_1\,-\,u_2$ and $\delta m\,:=\,m_1\,-\,m_2$  belong to
 $W^{1,2}\big([0,T];L^2(\R^2)\big)$;
\item $\nabla u_2$ and $\nabla m_2$ both belong to $L^1\big([0,T];L^\infty(\R^2)\big)$.
\end{enumerate}

Then, the following estimate holds true, for any $t\in[0,T]$:
\begin{align*}
 &\left\|\delta u(t)\right\|^2_{L^2}\,+\,\left\|\delta m(t)\right\|^2_{L^2}\,+\,\int^t_0\left\|\nabla\de m(\t)\right\|^2_{L^2}\,\dd\t\,
\leq\,C\, \Big(\left\|\delta u(0)\right\|^2_{L^2}\,+\,\left\|\delta m(0)\right\|^2_{L^2}\Big)\,e^{C\,\wtilde{\mc J}(t)}\,,
\end{align*}
where the function $\wtilde{\mc J}(t)$ is defined as
\[
\wtilde{ \mc J}(t)\,:=\,\int^t_0\Big(1\,+\,\left\|\nabla u_2(\t)\right\|_{L^\infty}\,+\,\left\|\nabla m_2(\t)\right\|_{L^\infty}\Big)\,\dd\t\,.
\]

\end{prop}

\begin{proof}
We only need to repeat the computations performed in the proof of Proposition \ref{p:stab}, but restricting to the case $p=2$.
First of all, an energy estimate for equation \eqref{eq:de-u}, which is well-justified under assumption (i) (see \tsl{e.g.} Chapter 5 of \cite{Evans}),
yields, for any $t\in[0,T]$, the bound
\begin{align*}
\left\|\de u(t)\right\|_{L^2}^2\,\lesssim\,\left\|\de u(0)\right\|_{L^2}^2\,+\,\int^t_0\Big(1\,+\,\left\|\nabla u_2(\t)\right\|_{L^\infty}\Big)\,
\left\|\de u(\t)\right\|_{L^2}^2\,\dd\t\,+\,\int^t_0\left\|\nabla\de m(\t)\right\|_{L^2}^2\,\dd\t\,,
\end{align*}
where we have also used the $L^2$-orthogonality between the pressure gradient and the divergence-free vector field $\de u$.

Next, from a standard energy estimate for equation \eqref{eq:de-m}, we infer that, for any $t\in[0,T]$, one has
the bound
\begin{align*}
&\left\|\de m(t)\right\|_{L^2}^2\,+\,\int^t_0\left\|\nabla\de m(\t)\right\|_{L^2}^2\,\dd\t \\
&\qquad\qquad\qquad
\lesssim \,
\left\|\de m(0)\right\|_{L^2}^2\,+\,\int^t_0\Big(\left\|\de u(\t)\right\|_{L^2}^2\,+\,\left\|\nabla m_2\right\|_{L^\infty}
\left\|\de u(\t)\right\|_{L^2}\,\left\|\de m(\t)\right\|_{L^2}\Big)\,\dd\t\,,
\end{align*}
where we have also used an integration by parts in order to pass the derivative from $\de\o$ to $\de m$, together with the Young inequality
to absorb the term $\nabla\de m$, arising from the previous computation, into the left-hand side. The computations are analogous to the ones
performed in Subsection \ref{ss:leb}, to arrive to inequality \eqref{est:u-m-energy}.

Summing up the previous bounds immediately yields the claimed inequality.
\end{proof}

\subsection{Uniqueness of solutions} \label{ss:unique}
We now apply the stability results of the previous subsection in order to deduce uniqueness of solutions to system \eqref{eq:2D-microp},
as claimed in Theorems \ref{th:global-strong} and \ref{th:endpoint}.

\subsubsection{The high regularity case} \label{sss:unique-general}

We discuss here the general case, in which we simply assume that the initial data have high regularity. Namely, we will suppose the index
$s\geq1+2/p$ in Theorem \ref{th:global-strong} to have a larger value, namely $s\geq 2+2/p$.

\begin{prop} \label{p:uni-gen}
Let $p\in\,]1,+\infty[\,$ and let $s\,\geq\,2\,\big(1\,+\,1/p\big)$. Take an initial datum $\big(u_0,m_0\big)$ such that
\[
 \div u_0\,=\,0\,,\qquad u_0\,\in\,B^s_{p,1}(\R^2)\qquad \mbox{ and }\qquad m_0\,\in\,B^{s-1}_{p,1}(\R^2)\,.
\]

Then there exists at most one solution $\big(u,m\big)$ to \eqref{eq:2D-microp} associated to this initial datum and defined on some time interval
$[0,T]$, with $T>0$, such that
\[
 u\,\in\,\mc C\big([0,T];B^s_{p,1}(\R^2)\big)\qquad \mbox{ and }\qquad m\,\in\,\mc C\big([0,T];B^{s-1}_{p,1}(\R^2)\big)\,\cap\,L^1\big([0,T];B^{s+1}_{p,1}(\R^2)\big)\,.
\]
\end{prop}

\begin{proof}
 The proof simply consists in verifying that, under the previous assumptions, the conditions in Proposition \ref{p:stab} are fulfilled, so that
we can apply the stability estimate claimed therein. This will imply local uniqueness, which then yields global uniqueness.

First of all, we notice that, under the above condition on $s$, namely $s\,\geq\, 2\,+\,2/p\,\geq\,2$, one has the embedding $B^{s}_{p,1}\,\hookrightarrow\, W^{1,p}$,
so that $u$ and $m$ in fact belong to $\mc C\big([0,T];W^{1,p}\big)$. Thus, the first part of condition (i) is satisfied.

Next, by the embedding $B^{s}_{p,1}\,\hookrightarrow\,W^{2,\infty}$ (which holds since $s-1\,\geq\,1+2/p$), we infer that conditions (ii) and (iii) are also fulfilled.
From this and the previously mentioned property, it follows in particular that
\[
(u\cdot\nabla)u\,,\ \o \ \mbox{ and }\ \nabla m\qquad \mbox{ all belong to }\quad \mc C\big([0,T];L^p\big)\,.
\]
By an inspection of equation \eqref{eq:pressure} for the pressure gradient, we thus deduce that also
$\nabla\Pi$ belongs to the same space $\mc C\big([0,T];L^p\big)$. This implies that
\[
 \d_tu\,=\,-\,(u\cdot\nabla)u\,-\,\nabla\Pi\,-\,\nabla^\perp m\;\in\,\mc C\big([0,T];L^p\big)\,.
\]
Likewise, since $s-2\,\geq\,2/p\,>\,0$ (which implies the embedding $B^{s-2}_{p,1}\,\hookrightarrow\,L^p$), we have that
\[
 \d_tm\,=\,-\,u\cdot\nabla m \,+\,\Delta m\,+\,\o\;\in\,\mc C\big([0,T];L^p\big)\,.
\]
So, the second part of assumption (i) is verified as well.

To conclude, we argue as in Remark \ref{r:hom-B} to check that also assumptions (iv) and (v) are satisfied, with $r=2$ (actually, here we have
$\nabla m\in L^\infty\big([0,T]\times\R^2\big)$ by Besov embeddings). As for condition (vi), it follows by embeddings again,
as the couple $(s-2,1)$ verifies the conditions in \eqref{eq:AnnLInfty}, so that $B^{s-2}_{p,1}\,\hookrightarrow\,L^\infty$ and one has
$\nabla\o\in L^\infty\big([0,T]\times\R^2\big)$.

In the end, we have verified that, under our assumptions, one can apply the stability estimates of Proposition \ref{p:stab}. This implies local uniqueness,
whence (as already pointed out at the beginning) global uniqueness on $[0,T]$ follows by standard arguments.
\end{proof}

\subsubsection{The case of lower regularity} \label{sss:unique-energy}

Here, we discuss the case of lower values of $s$, under suitable conditions over the integrability index $p$.
The counterpart of the Proposition \ref{p:uni-gen} in this case is given by the following statement.

\begin{prop} \label{p:uni-energy}
Let $p\in\,]1,+\infty]$ and let $s\,\geq\,1\,+\,2/p$. Take an initial datum $\big(u_0,m_0\big)$ such that the following
conditions are satisfied:
\[
 \div u_0\,=\,0\,,\qquad u_0\,\in\,B^s_{p,1}(\R^2)\qquad \mbox{ and }\qquad m_0\,\in\,B^{s-1}_{p,1}(\R^2)\,.
\]
Assume that one of the following conditions hold true:
\begin{enumerate}[\rm (i)]
\item $1<p\leq 2$;
\item $p\in\,]2,+\infty]$, and one further supposes either that both $u_0$ and $m_0$ belong to $L^2(\R^2)$,
or that there exists $1<p_0<2$ such that $\o_0$ and $m_0$ belong to $L^{p_0}(\R^2)$.
\end{enumerate}

Then there exists at most one solution $\big(u,m\big)$ to \eqref{eq:2D-microp} associated to this initial datum and defined on some time interval
$[0,T]$, with $T>0$, such that $u$ and $m$ satisfy the regularity conditions of Theorem \ref{th:global-strong}, or of Theorem \ref{th:endpoint} in
the case of assumption {\rm (ii)} and integrability conditions over $\o_0$ and $m_0$.
%
\end{prop}

Since, by Sobolev embeddings, it is clear that assumption (ii) in Proposition \ref{p:stab-energy} is always satisfied under the hypothesis
$s\geq1+2/p$, the proof of Proposition \ref{p:uni-energy} boils down to showing the next lemma.

\begin{lemma} \label{l:energy}
Let $p\in\,]1,+\infty]$ and let $s\,\geq\, 1\,+\,2/p$.
Take two sets $\big(u_{0,j}, m_{0,j}\big)_{j=1.2}$ of initial data satisfying the assumptions of Theorem \ref{th:global-strong}.
Assume moreover that one of the following two conditions are satisfied:
\begin{itemize}
\item either $1<p\leq 2$;
\item or $2<p\leq+\infty$ and both $u_{0,j}$ and $m_{0,j}$ belong to $L^2(\R^2)$, for $j=1,2$;
\item or $2<p\leq+\infty$ and one further supposes that there exists $1<p_0<2$ such that $\o_0$ and $m_0$ belong to $L^{p_0}(\R^2)$, and that
one has the properties $\de u_0\,:=\,u_{0,1}\,-\,u_{0,2}\;\in\,L^2(\R^2)$ and $\de m_0\,:=\,m_{0,1}\,-\,m_{0,2}\;\in\,H^2(\R^2)$;
\end{itemize}
Let $\big(u_j,m_j\big)_{j=1,2}$ be respective solutions related to those sets of initial data defined on some common time interval $[0,T]$,
and assume that they satisfy the regularity conditions listed in Theorem \ref{th:global-strong} in the case $p\in\,]1,+\infty[\,$,
or the ones listed in Theorem \ref{th:endpoint} in the case $p=+\infty$.

Then the differences $\de u\,:=\,u_1-u_2$ and $\de m\,:=\,m_1-m_2$ belong to $W^{1,r}\big([0,T];L^2(\R^2)\big)$, for any index
$r\in[2,+\infty[\,$.
\end{lemma}

\begin{proof}
Let us start our discussion by treating the easiest case: we observe that, when $p=2$, we can proceed as done in the proof of Proposition
\ref{p:uni-gen} above. This yields that, for $j=1,2$, both $\d_tu_j$ and $\d_tm_j$ belong to
$\mc C\big([0,T];L^2\big)$. Then, the result easily follows, as one can even take $r=+\infty$ in this case.

The same argument directly applies whenever the initial data $u_{0,j}$ and $m_{0,j}$ belong to $L^2(\R^2)$, for $j=1,2$. Indeed, under these assumptions,
the energy estimate \eqref{est:u-m-energy} yields that, for any $j=1,2$, one has $u_j\in\mc C\big([0,T];L^2\big)$ and
$m_j\in \mc C\big([0,T];L^2\big)\,\cap\,L^2\big([0,T];H^1\big)$ and we can conclude as before.
As a matter of fact, we point out that the property $\nabla\Pi_j \in L^\infty\big([0,T];L^2\big)$, necessary to perform energy estimates for $u_j$,
holds true also in the case $p=+\infty$, under the finite energy condition over $u_{0,j}$: for this, it is enough to resort to
equation \eqref{eq:pressure} and use that $u_j\in L^\infty_T(L^2)$ and $\nabla u_j\in L^\infty_T(L^\infty)$.
In particular, this argument applies when $1<p<2$, or when $2<p\leq+\infty$ and one assumes $u_{0,j},m_{0,j}\,\in\,L^2(\R^2)$, for $j=1,2$.

Next, let us focus on the case $2<p\leq+\infty$, with the assumption that 
$\o_0$ and $m_0$ belong to $L^{p_0}(\R^2)$, for some $p_0\in\,]1,2[\,$, and that 
$\de u_0\,\in\,L^2(\R^2)$ and $\de m_0\,\in\,H^2(\R^2)$.
By Besov embeddings, we see that one has
\[
 \mbox{ for }\ j\,=\,1,2\,,\qquad\qquad \o_{0,j}\,,\,m_{0,j}\;\in\,L^{p_0}(\R^2)\cap L^\infty(\R^2)\,,
\]
so, in particular, they belong to $L^2(\R^2)$ also. Thus, by arguing as for getting \eqref{est:G-m_final} and \eqref{est:o_final},
see also Remark \ref{r:m-L^2}, one has, for each $j=1,2$, that
\begin{equation} \label{eq:prop_o-m_j}
\o_j\,,\, m_j\;\in\,\mc C\big([0,T];L^2\big)\,,\qquad\qquad \mbox{ with }\qquad \nabla m_j\,\in\,L^2\big([0,T];L^2\big)\,.
\end{equation}
The difficulty, here, relies in treating the difference of the velocity fields $\de u$, because from our assumptions it does not follow directly that each
$u_j$ belongs to $L^2$. To get the sought property, we will argue as in the proof of the classical Yudovich theorem (see \tsl{e.g.} Chapter 8 of \cite{Maj-Bert},
see also \cite{F-FD}). Consider equation \eqref{eq:de-u} for $\de u$ and the corresponding one for the difference of the pressure gradients
$\nabla\de\Pi\,:=\,\nabla\Pi_1\,-\,\nabla\Pi_2$, which we recall here for the reader's convenience:
\begin{align}
\label{eq:de-u_uniq}
& \d_t\de u\,+\,(u_1\cdot\nabla)\de u\,+\,(\de u\cdot\nabla)u_2\,+\,\nabla\de\Pi\,=\,-\,\nabla^\perp\de m \\
\nonumber
&\qquad \mbox{ and }\qquad 
\nabla\de\Pi\,=\,\nabla(-\Delta)^{-1}\div\Big((u_1\cdot\nabla)\de u\,+\,(\de u\cdot\nabla)u_2\Big)\,.
\end{align}
By Besov embeddings, we know that each $u_j$ belongs to the space $\mc C_T(L^\infty)$, which implies, together with \eqref{eq:prop_o-m_j}, that the bilinear
terms $(u_1\cdot\nabla)\de u$ and $(\de u\cdot\nabla)u_2$ both belong to the space $\mc C_T(L^2)$. From the elliptic equation for $\nabla\de\Pi$,
one deduces that $\nabla\de\Pi$ belongs to the same space $\mc C_T(L^2)$. Since $\nabla^\perp\de m$ belongs to $L^2_T(L^2)$ by \eqref{eq:prop_o-m_j} and the initial 
datum $\de u_0$ belongs to $L^2$ by assumption, we infer that
\[
 \de u\,\in\,\mc C\big([0,T];L^2\big)\,.
\]
Next, we consider the equation for $\de m$, which we write in the form
\begin{equation} \label{eq:de-m_heat}
 \d_t\de m\,-\,\Delta\de m\,=\,-\,u_1\cdot\nabla \de m\,-\,\de u\cdot\nabla m_2\,+\,\de\o\,,
\end{equation}
with initial datum $\de m_{0}\in H^2$. For later use, we observe that, for any $r\in[2,+\infty]$, one has $H^2\equiv B^2_{2,2}\,\hookrightarrow\,B^{2-\frac{2}{r}}_{2,r}\,
\hookrightarrow\,\dot B^{2-\frac{2}{r}}_{2,r}$, since $2-\frac{2}{r}>0$. At the same time, the previous argument yields that
$u_1\cdot\nabla \de m\,,\,\de u\cdot\nabla m_2\;\in\,L^2_T(L^2)$ and $\de \o\,\in\,\mc C_T(L^2)$.
Thus, we can start by applying a maximal regularity estimate (similarly to what done in the proof of Proposition \ref{p:stab}) to deduce that
\[
 \d_t\de m\,,\,\Delta\de m\;\in\,L^2\big([0,T];L^2\big)\,,
\]
which implies that $\de m\in \mc C_T(L^2)\cap L^2_T(H^2)$. By interpolation, we thus find that $\nabla \de m\in L^4_T(L^2)$.
Inserting this property into \eqref{eq:de-m_heat}, we further infer that
\[
 \d_t\de m\,,\,\Delta\de m\;\in\,L^4\big([0,T];L^2\big)\,,
\]
because each $u_j$ belongs to $\mc C_T(L^\infty)$, each $\o_j$ belongs to $\mc C_T(L^p)$ for any $p\in[p_0,+\infty]$ and each $\nabla m_j$ belongs to
$L^\infty_T(L^\infty)$ by Besov embeddings.
By iterating the previous argument, one can thus prove that
\[
 \d_t\de m\,,\,\Delta\de m\;\in\,\bigcap_{2\leq r<+\infty}L^r\big([0,T];L^2\big)\,.
\]
Unfortunately, the multiplicative constants coming from maximal regularity estimates are not uniform in $r\geq2$,
so we cannot pass to the limit $r\to+\infty$ in the above property. However, from it we still deduce that
$\de m\in W^{1,r}\big([0,T];L^2\big)\cap L^r\big([0,T];H^2\big)$, for any $r\in[2,+\infty[\,$.
At this point, going back to \eqref{eq:de-u_uniq} and using this property, we find that $\d_t\de u$ belongs to $L^r_T(L^2)$ for any finite $r\geq2$.
The proof of the lemma is thus completed.
\end{proof}

We are now in the position of proving Proposition \ref{p:uni-energy}.

\begin{proof}[Proof of Proposition \ref{p:uni-energy}]
The proof is simply based on the application of Proposition \ref{p:stab-energy}. Thus, we only need to check that the assumptions of that result are fulfilled.

As already pointed out, the condition in item (ii) of Proposition \ref{p:stab-energy} simply follows from the fact that
$\nabla u\,\in\,L^\infty_T(B^{s-1}_{p,1})$ and (by interpolation) $\nabla m\,\in\,L^2_T(B^{s-1}_{p,1})$. Indeed,
as $s-1\geq 2/p\geq 0$, one has $B^{s-1}_{p,1}\,\hookrightarrow\,L^\infty$.

We have now to check the condition in item (i) of Proposition \ref{p:stab-energy}. However, this is a straightforward consequence of
Lemma \ref{l:energy}. Indeed, whenever $1<p\leq2$ (since $B^{s-1}_{p,1}\,\hookrightarrow\,L^\infty$ under our assumptions
over $s$) or $2<p\leq+\infty$ and one further assumes that $u_0\,,\,m_0\;\in\,L^{2}(\R^2)$, a direct application of Lemma \ref{l:energy} yields the conclusion.
Focus now on the hypothesis $2<p\leq+\infty$ and that there exists $p_0\in\,]1,2[\,$ such that $\o_0$ and $m_0$ belong to $L^{p_0}(\R^2)$.
Since the difference $\big(\de u, \de m\big)$ of two hypothetical solutions related to the same initial datum would have zero initial value,
we can argue as in the third item of Lemma \ref{l:energy} to deduce that $\de u\,,\,\de m\;\in W^{1,r}\big([0,T];L^2\big)$ for any $r\in[2,+\infty[\,$.
Then, assumption (i) of Proposition \ref{p:stab-energy} holds true also in this instance, as claimed.

In the end, we have verified that the assumptions of Proposition \ref{p:uni-energy} allow to apply the stability estimates from Proposition \ref{p:stab-energy}.
Therefore, uniqueness follows.
\end{proof}

\section{Proof of existence: local theory} \label{s:existence}

In this section, we prove the existence of solutions at the level of regularity claimed in Theorems \ref{th:global-strong} and \ref{th:endpoint}.
However, we will show here only \emph{local} existence. The global character of these solutions (under suitable assumptions on $s$)
will be established in Section \ref{s:global} below.

While the scheme of the proof is based on the classical three-step argument (regularisation of the initial data, construction of smooth approximate solutions, convergence
to a solution of the original system), the actual construction of smooth approximate solutions will depend on the value of the regularity index $s$.
This will entail changes also in the convergence argument, which relies on the stability estimates from Subsection \ref{ss:stab} for large values of $s$
(or suitable ranges of $p\in\,]1,+\infty]$), on a compactness argument for smaller values of $s$ (and general $p$).

For convenience of notation, throughout this section we will adopt the following convention. 
Given a sequence of functions $\big(f_n\big)_{n\in\N}\,\subset\,\mf B$, where $\mf B$ is a Banach space,
we will write $\big(f_n\big)_{n\in\N}\sqsubset \mf B$ if there exists a constant $C>0$ such that $\|f_n\|_{\mf B}\leq C$ for all $n\in\N$.

\subsection{Regularisation of the initial datum} \label{ss:regul}

We start here by presenting the regularisation procedure of the initial datum. This step does not actually depend on the specific assumptions and is common
to both Theorems \ref{th:global-strong} and \ref{th:endpoint}, with no further restrictions on $s$ nor on $p$.

So, let $\big(u_0,m_0\big)$ be the initial datum fixed in the statement of either Theorem \ref{th:global-strong},
with possibly the additional integrability conditions stated in (G.2) or (G.3), or Theorem \ref{th:endpoint}.
For any $n\in\N$, we smooth out this initial datum by setting
\[
 u_{0,n}\,:=\,S_n u_0\qquad\quad \mbox{ and }\qquad\quad m_{0,n}\,:=\,S_n m_0\,,
\]
where the operators $S_n$ are the low-frequency cut-off operators associated to a Littlewood-Paley decomposition, as defined in \eqref{eq:S_j}.
We thus obtain a sequence of smooth functions  $\big(u_{0,n},m_{0,n}\big)_{n\in\N}$, which in addition possess suitable integrability properties
together with all their derivatives. As these properties depend on the value of $p\in\,]1,+\infty]$ and on the additional integrability assumptions,
for the sake of conciseness we avoid to give the details here.

It immediately follows from the definitions and the properties of the operators $S_n$ that
\begin{align*}
\forall\,n\in\N\,,\quad &\div u_{0,n}\,=\,0\,, \\
& \qquad \mbox{ and }\qquad \big(u_{0,n}\big)_{n\in\N}\,\sqsubset\,B^{s}_{p,1}\;,\qquad \big(m_{0,n}\big)_{n\in\N}\,\sqsubset\,B^{s-1}_{p,1}\,.
\end{align*}
In addition, we have the following strong convergence properties (see \tsl{e.g.} Lemma 2.73 of \cite{BCD}):
in the limit for $n\to+\infty$, one has
\begin{align*}
& u_{0,n}\, \longrightarrow\,u_0 \qquad  \text{ in } \quad B^s_{p,1} \qquad\qquad \mbox{ and }\qquad\qquad
m_{0,n}\, \longrightarrow\,u_0 \qquad  \text{ in } \quad B^{s-1}_{p,1}\,.
\end{align*}
Once again, we point out that, in the presence of further assumptions on $\big(u_0,m_0\big)$ and the initial vorticity
$\o_0$,
like in conditions (G.2) and (G.3) of Theorem \ref{th:global-strong} or in Theorem \ref{th:endpoint}, one also gets convergence of the approximations
to the original data in the strong topology of the related Lebesgue spaces.


\medbreak
For any $n\in\N$, we are now going to construct a unique smooth solution associated to the initial datum $\big(u_{0,n}, m_{0,n}\big)$.
However, due to reasons which will appear clear later on, we need to split our discussions into two cases. In Subsection \ref{ss:smooth-s}
we treat the case of high regularity and subcritical values. In Subsection \ref{ss:smooth-less}, instead, we consider the case of lower values of
the regularity index $s$, and the critical regularity cases.

\subsection{Smooth approximate solutions: high regularity} \label{ss:smooth-s}
We now construct smooth approximate solutions associated to the regularised initial data, under one of the following specific assumptions:
\begin{itemize}
 \item $p\in\,]1,+\infty]$ and $s\,>\, 1\,+\,2/p$, together with the assumption $u_0,m_0\;\in L^2(\R^2)$
(which simply follows by embeddings in the case $1< p\leq 2$);
 \item $p\in\,]2,+\infty[\,$ and $s\,\geq\,2\,\big(1+1/p\big)$ (but the initial datum does not necessarily have finite energy);
  \item $p=+\infty$ and $s> 1$, with in addition $\o_0$ and $m_0$ belong to $L^{p_0}(\R^2)$ (as in the assumptions of Theorem \ref{th:endpoint}).
\end{itemize}
It is important to stress that we avoid to use this construction in the case of critical values of $s$, namely when $s=1+2/p$.
The reason will be apparent in Section \ref{s:global}, see also the discussion at the beginning of Subsection \ref{ss:smooth-less} below.

Before going on, let us remark that, of course, in presence of finite energy initial data (first item of the previous list) we could use a Friedrichs
method for the construction of smooth
approximate solutions, a method which would be easier than the one we are going to present here. We do not do this here, for the sake of conciseness. 
On the contrary, we put the accent on the convergence of the approximation scheme: as a matter of fact, in all of the three cases above, we can use the
stability estimates of Subsection \ref{ss:stab}.

\subsubsection{Construction} \label{sss:construct}

We treat the three cases above in a unified way and proceed by induction.
In order to avoid confusion, the first element of the iterative scheme will be denoted with the index $1$. For convenience of notation,
also in this part we set $\alpha=\k=1$.

To begin with,  we first set
\[
u^1\,=\,u_{0,1} \qquad \mbox{ and }\qquad m^1\,=\,m_{0,1}\,.
\]
We also define $\o^1\,=\,\curl(u^1)\,=\,\o_{0,1}$ and
\begin{equation*}
\Gamma^1\,:=\,\omega^1\,+\, m^1\,.
\end{equation*}

Next, let us assume that, for some $n\in\N$, we have already defined smooth functions  
$\big(u^n, m^n\big)$ such that $u^n$ and $m^n$ belong, respectively, to $\mc C\big(\R_+;B^s_{p,1}\big)$ and
$\mc C\big(\R_+;B^{s-1}_{p,1}\big)\,\cap\,L^1_\loc\big(\R_+;B^{s+1}_{p,1}\big)$,
together with the property $\div u^n\,=\,0$.
In the case of finite energy initial data, we also assume
$u^n\in \mc C\big(\R_+;L^2\big)$ and $m^n\in\mc C\big(\R_+;L^2\big)\,\cap\,L^2_\loc\big(\R_+;H^1\big)$. In the case of additional integrability
conditions over $\o_0$ and $m_0$, we assume instead that
$\o^n\,:=\,\curl(u^n)$ and $m^n$ both belong to $\mc C\big(\R_+; L^{p_0}\cap L^\infty\big)$.
It is an easy verification to see that the first element $\big(u^1,m^1\big)$ constructed above satisfies these properties.

In order to construct the $(n+1)$-th term of the sequence, we start by defining
$m^{n+1}$ as the solution of the linear transport-diffusion equation
\begin{equation}
\label{eq:def_m_n}
    \d_tm^{n+1}\,+\,u^n\cdot\nabla m^{n+1}\,-\,\Delta m^{n+1}\,=\,\o^n,
\end{equation}
with initial datum
\[
m^{n+1}_{|{t=0}}\,=\,m_{0,n+1}\,.
\]
Since the transport field $u^n$ and the source term $\o^n$ are smooth, thanks to \tsl{e.g.} Theorem \ref{t:tr-diff} we get that,
for the given smooth initial datum $m_{0,n+1}$ there exists a unique smooth solution of the transport diffusion equation \eqref{eq:def_m_n} such that
\begin{equation} \label{eq:reg_m-n}
m^{n+1}\;\in \mc C\big(\R_+;B^{s-1}_{p,1}\big)\,\cap\,L^1_\loc\big(\R_+;B^{s+1}_{p,1}\big)\,.
\end{equation}
Next, we define the couple $\big(u^{n+1},\nabla\Pi^{n+1}\big)$ as the unique smooth solution of the linear hyperbolic-elliptic system
\begin{equation} \label{eq:def_u_n}
\left\{\begin{array}{l}
        \d_tu^{n+1}\,+\,(u^n\cdot\nabla)u^{n+1}\,+\,\nabla\Pi^{n+1}\,=\,-\,\nabla^\perp m^{n+1} \\[1ex]
        -\Delta\Pi^{n+1}\,=\,\div\big((u^n\cdot\nabla)u^{n+1}\big)\,,
       \end{array}
\right.
\end{equation}
associated to the initial datum
\[
 u^{n+1}_{|t=0}\,=\,u_{0,n+1}\,.
\]
Observe that applying the divergence operator to the equation for $u^{n+1}$, by definition of $\nabla\Pi^{n+1}$ we immediately see that
\[
 \d_t\div u^{n+1}\,=\,0\,,\qquad\qquad \mbox{ which implies }\qquad \div u^{n+1}\,=\,0\quad \mbox{ on }\ \R_+\times\R^2\,,
\]
as the initial datum is divergence-free.
On the other hand, if we apply the $\curl$ operator of equation \eqref{eq:def_u_n}, we find an equation for $\o^{n+1}\,:=\,\curl(u^{n+1})$:
\begin{equation} \label{eq:o-n}
\d_t\o^{n+1}\,+\,u^n\cdot\nabla \o^{n+1}\,=\,-\,\Delta m^{n+1}\,+\,\mc B\big(\nabla u^n,\nabla u^{n+1}\big)\,,
\end{equation}
where, given two divergence-free vector fields $v$ and $w$, we have defined the bilinear (skew-symmetric) operator $\mc B(\nabla v,\nabla w)$ by the formula
\[
 \mc B(\nabla v,\nabla w)\,:=\,\d_1w^1\,\big(\d_1v^2\,+\,\d_2v^1\big)\,-\,\d_1v^1\,\big(\d_1w^2\,+\,\d_2w^1\big)\,.
\]
Thus, after defining $\Gamma^{n+1}$ by the analogue of formula \eqref{eq:def_G} applied to $m^{n+1}$ and $\o^{n+1}$, we see that it solves the transport
equation
\begin{equation}
\label{eq:G_n}
\d_t\Gamma^{n+1}\,+\,u^n\cdot\nabla\Gamma^{n+1}\, 
=\,\o^n\,+\,\mc B\big(\nabla u^n,\nabla u^{n+1}\big)\,,
\end{equation}
supplemented with (recall that $\alpha=\k=1$ here) the initial condition
\[
\Gamma^{n+1}_{|{t=0}}\,=\,\Gamma_{0,n+1}\,:=\,\omega_{0,n+1}\,+\,m_{0,n+1}\,.
\]
Now, using Theorem \ref{th:transport} for \eqref{eq:o-n} together with \eqref{eq:reg_m-n} and the $L^p$ bounds developed in Subsection \ref{ss:leb},
it is easy to prove that
\[
 u^{n+1}\;\in\,\mc C\big(\R_+;B^s_{p,1}\big)\,.
\]

Finally, we also mention that, in case of finite energy initial datum, we can show in addition that 
$u^{n+1}\in \mc C\big(\R_+;L^2\big)$ and $m^{n+1}\in\mc C\big(\R_+;L^2\big)\,\cap\,L^2_\loc\big(\R_+;H^1\big)$. In the case of additional integrability
conditions over $\o_0$ and $m_0$, we also have that
$\o^{n+1}$ and $m^{n+1}$ both belong to the space $\mc C\big(\R_+; L^{p_0}\cap L^\infty\big)$.

\subsubsection{Uniform bounds} \label{sss:unif-bounds}

The above mentioned regularity properties for the sequence of solutions $\big(u^n,m^n\big)_{n\geq1}$ are not enough for our scopes.
Indeed, as is well-known, uniform bounds are needed for passing to the limit $n\to+\infty$: getting these bounds is precisely the scope of this paragraph.
For simplicity, we only focus on the second case among the three listed at the beginning of Subsection \ref{ss:smooth-s}. Namely,
throughout this part we are going to assume that
\[
 p\,\in\;]2,+\infty[\;\qquad\qquad \mbox{ and }\qquad\qquad s\,\geq\, 2\,\left(1\,+\,\frac{1}{p}\right)\,.
\]
The other two cases can be treated by very similar arguments, with the price of estimating some additional Lebesgue norm
of the solutions.

Here, the computations are analogous to the ones performed in Section \ref{s:a-priori}, hence let us revisit them quickly. First of all,
for $p$ and $s$ fixed as above, we define, for any integer $n\geq 1$ and any $t\geq0$, the $n$-th energy functional
\begin{equation} \label{eq:def_E_n}
 \mc E_n(t)\,:=\,\left\|u^n\right\|_{L^\infty_t(B^{s}_{p,1})}\,+\,\left\|m^n\right\|_{L^\infty_t(B^{s-1}_{p,1})\cap L^1_t(B^{s+1}_{p,1})}\,,
\end{equation}
where we agree that, for $t=0$, the quantity $\mc E_n(0)$ is defined as 
\[
 \mc E_n(0)\,:=\,\left\|u_{0,n}\right\|_{B^{s}_{p,1}}\,+\,\left\|m_{0,n}\right\|_{B^{s-1}_{p,1}}\,.
\]
Observe that, owing to the properties of the regularising operators, one has
\begin{equation} \label{est:u-b_E_0}
\forall\,n\,\in\,\N\,,\ n\geq1\,,\qquad\qquad \mc E_{n}(0)\,\lesssim\,\mc E(0)\,:=\,\left\|u_{0}\right\|_{B^{s}_{p,1}}\,+\,\left\|m_{0}\right\|_{B^{s-1}_{p,1}}\,.
\end{equation}
Remark also that, thanks to \eqref{est:low-high}, for any $t\geq0$ one has the equivalence
\[
\mc E_n(t)\,\approx\,\left\|u^n\right\|_{L^\infty_t(L^p)}\,+\,\left\|\o^n\right\|_{L^\infty_t(B^{s-1}_{p,1})}\,+\,
\left\|m^n\right\|_{L^\infty_t(B^{s-1}_{p,1})\cap L^1_t(B^{s+1}_{p,1})}\,.
\]

Next, we bound the Lebesgue norms of the solution, following the main steps performed in Subsection \ref{ss:leb}. The point to be noted here is that,
differently from those computations, the estimates we get are not global in time, essnentially due to the presence of the additional bilinear term $\mc B$
in \eqref{eq:o-n}.

Performing a $L^p$ estimate for system \eqref{eq:def_u_n}, since $p\in\,]2,+\infty[\,$ is finite, we get
\begin{align*}
 \left\|u^{n+1}\right\|_{L^\infty_t(L^p)}\,&\lesssim\,\left\|u_{0,n+1}\right\|_{L^p}\,+\,\int^t_0\Big(\left\|\nabla^\perp m^{n+1}\right\|_{L^p}\,+\,
 \left\|\nabla\Pi^{n+1}\right\|_{L^p}\Big)\,\dd\t \\
 &\lesssim\,\left\|u_{0,n+1}\right\|_{L^p}\,+\,\int^t_0\Big(\left\|m^{n+1}\right\|_{W^{1,p}}\,+\,
\left\|u^{n}\right\|_{L^\infty}\, \left\|\nabla u^{n+1}\right\|_{L^p}\Big)\,\dd\t\,.
\end{align*}
By using the embeddings $B^s_{p,1}\,\hookrightarrow\,L^\infty$ and $B^{s-1}_{p,1}\,\hookrightarrow\,W^{1,p}$, which hold true under our assumptions
(for the letter, we use the fact that $s-1\,\geq\,1\,+\,2/p\,\geq\,1$), we can thus bound
\begin{equation} \label{est:u_n-L^p}
 \left\|u^{n+1}\right\|_{L^\infty_t(L^p)}\,\lesssim\,\mc E_{n+1}(0)\,+\,\int^t_0\Big(1\,+\,\mc E_n(\t)\Big)\,\mc E_{n+1}(\t)\,\dd\t\,.
\end{equation}

Next, repeating \tsl{mutatis mutandis} the computations of Subsection \ref{ss:leb}, we can bound the $L^\infty$ norm of $m^{n+1}$ and $\Gamma^{n+1}$
as follows:
\begin{align*}
\left\|m^{n+1}\right\|_{L^\infty_t(L^\infty)}\,&\lesssim\,\left\|m_{0,n+1}\right\|_{L^\infty}\,+\,\int^t_0\left\|\o^{n}\right\|_{L^\infty}\,\dd\t\,, \\
\left\|\Gamma^{n+1}\right\|_{L^\infty_t(L^\infty)}\,&\lesssim\,\left\|\Gamma_{0,n+1}\right\|_{L^\infty}\,+\,
\int^t_0\Big(\left\|\o^{n}\right\|_{L^\infty}\,+\,\left\|\nabla u^n\right\|_{L^\infty}\,\left\|\nabla u^{n+1}\right\|_{L^\infty}\Big)\,\dd\t\,.
\end{align*}
Putting them together, those bounds give rise to the estimate
\begin{align}
\label{est:m-G_n-inf}
\left\|\Big(\o^{n+1}\,,\,m^{n+1}\Big)\right\|_{L^\infty_t(L^\infty)}\,&\lesssim\,\mc E_{n+1}(0)\,+\,\int^t_0\mc E_n(\t)\,\dd\t\,+\,
\int^t_0\mc E_n(\t)\,\mc E_{n+1}(\t)\,\dd\t\,,
\end{align}
where, for simplicity of notation, we have set $\|(f,g)\|_X\,=\,\|f\|_X\,+\,\|g\|_X$.

The next step consists in bounding the Besov norms of $\o^{n+1}$ and $m^{n+1}$. 
As before, treating the vorticity function relies on the use of the auxiliary unknown $\Gamma^{n+1}$: applying Theorem \ref{th:transport}
to equation \eqref{eq:G_n} and using the tame estimates of Corollary \ref{c:tame} and embeddings, we get
\begin{align*}
\left\|\Gamma^{n+1}\right\|_{L^\infty_t(B^{s-1}_{p,1})}\,&\lesssim\,\left\|\Gamma_{0,n+1}\right\|_{B^{s-1}_{p,1}}\,+\,
\int^t_0\Big(\left\|\o^n\right\|_{B^{s-1}_{p,1}}\,+\,\left\|\nabla u^n\right\|_{B^{s-1}_{p,1}}\,\left\|\nabla u^{n+1}\right\|_{B^{s-1}_{p,1}}\Big)\,\dd\t \\
&\qquad\qquad\qquad
\,+\,\int^t_0\Big(\left\|\nabla u^n\right\|_{L^\infty}\,\left\|\Gamma^{n+1}\right\|_{B^{s-1}_{p,1}}\,+\,
\left\|\Gamma^{n+1}\right\|_{L^\infty}\,\left\|\nabla u^n\right\|_{B^{s-1}_{p,1}}\Big)\,\dd\t\,,
\end{align*}
where we have argued as in Subsection \ref{ss:bes} to treat the commutator coming from the transport term $u^n\cdot\nabla\Gamma^{n+1}$.
Thus, we deduce that
\begin{align*}
\left\|\Gamma^{n+1}\right\|_{L^\infty_t(B^{s-1}_{p,1})}\,&\lesssim\,\mc E_{n+1}(0)\,+\,\int^t_0\mc E_n(\t)\,\dd\t\,+\,
\int^t_0\mc E_n(\t)\,\mc E_{n+1}(\t)\,\dd\t\,.
\end{align*}
The estimate of the Besov norm of $m^{n+1}$ is based on the application of Theorem \ref{t:tr-diff} to the transport-diffusion equation \eqref{eq:def_m_n}:
we find that
\begin{align*}
\left\|m^{n+1}\right\|_{L^\infty_t(B^{s-1}_{p,1})\cap L^1_t(B^{s+1}_{p,1})}\,&\lesssim\,(1+t)
\left(\left\|m_{0,n+1}\right\|_{B^{s-1}_{p,1}}+
\int^t_0\Big(\left\|\o^n\right\|_{B^{s-1}_{p,1}}+\left\|u^n\cdot\nabla m^{n+1}\right\|_{B^{s-1}_{p,1}}\Big)\dd\t\right) \\
&\lesssim\,(1+t)\,\left\|m_{0,n+1}\right\|_{B^{s-1}_{p,1}}\,+\,(1+t)\,\int^t_0\left\|\o^n\right\|_{B^{s-1}_{p,1}}\,\dd\t \\
&\qquad \,+\,(1+t)\,
\int^t_0\Big(\left\|u^n\right\|_{L^\infty}\,\left\|m^{n+1}\right\|_{B^{s}_{p,1}}\,+\,
\left\|u^n\right\|_{B^s_{p,1}}\,\left\|m^{n+1}\right\|_{L^\infty}\Big)\,\dd\t\,,
\end{align*}
where, for passing from the first inequality to the second one, we have argued as in Subsection \ref{ss:bes}.
By the use of an interpolation argument together with the Young inequality, we are thus led to the following bound:
\begin{align*}
\left\|m^{n+1}\right\|_{L^\infty_t(B^{s-1}_{p,1})\cap L^1_t(B^{s+1}_{p,1})}\,&\lesssim\,(1+t)\,\mc E_{n+1}(0)\,+\,(1+t)^2\,
\int^t_0\left\|u^n\right\|^2_{L^\infty}\,\left\|m^{n+1}\right\|_{B^{s-1}_{p,1}}\,\dd\t \\
&\qquad\qquad \,+\,(1+t)\,\int^t_0\mc E_n(\t)\,\dd\t
\,+\,(1+t)\,\int^t_0\mc E_n(\t)\,\mc E_{n+1}(\t)\,\dd\t\,.
\end{align*}
Putting the estimate for $\Gamma^{n+1}$ with the previous inequality, we infer that
\begin{align}
 \label{est:o-m_n-B}
&\left\|\o^{n+1}\right\|_{L^\infty_t(B^{s-1}_{p,1})}\,+\,\left\|m^{n+1}\right\|_{L^\infty_t(B^{s-1}_{p,1})\cap L^1_t(B^{s+1}_{p,1})} \\
\nonumber
&\quad \lesssim\,(1+t)\,\mc E_{n+1}(0)\,+\,(1+t)\,\int^t_0\mc E_n(\t)\,\dd\t\,+\,
(1+t)^2\,\int^t_0\Big(1\,+\,\left(\mc E_n(\t)\right)^2\Big)\,\mc E_{n+1}(\t)\,\dd\t\,.
\end{align}

Finally, we can gather estimates \eqref{est:u_n-L^p}, \eqref{est:m-G_n-inf} and \eqref{est:o-m_n-B} to deduce that, for any $t\geq0$, one has
\begin{align}
\label{est:E_n-tot}
 \mc E_{n+1}(t)\,&\lesssim\,(1+t)\,\mc E(0)\,+\,(1+t)\,\int^t_0\mc E_n(\t)\,\dd\t \\
\nonumber
&\qquad\qquad\qquad\qquad\qquad
\,+\,(1+t)^2\,\int^t_0\Big(1\,+\,\left(\mc E_n(\t)\right)^2\Big)\,\mc E_{n+1}(\t)\,\dd\t\,,
\end{align}
where we have also used the uniform bound in \eqref{est:u-b_E_0} for the initial datum.

\medbreak
With inequality \eqref{est:E_n-tot} at hand, we can easily find uniform estimates (with respect to $n\geq1$) on some finite time interval $[0,T]$. We proceed
similarly as done in Subsection \ref{ss:bes}. Define the function $f(t)=(1+t)^2$ and, for any integer $n\geq1$, the time
\[
 T_n\,:=\,\sup\left\{t\,>\,0\;\bigg|\quad \int^t_0\Big(1\,+\,\left(\mc E_n(\t)\right)^2\Big)\,\dd\t\,\leq\,2\,\min\big\{\log 2\,,\,\mc E(0)\big\}\right\}\,.
\]
Observe that, for any $n\geq1$, the time $T_n$ is well-defined and one has $T_n>0$.
Moreover, by \eqref{est:E_n-tot} and an application of the Gr\"onwall lemma, we see that, in the time interval $[0,T_n]$, one has the estimate
\begin{equation} \label{est:u_b-E_n_part}
  \mc E_{n+1}(t)\,\leq\,C\,f(t)\,e^{C\,f(t)}\,\mc E(0)\,,
\end{equation}
for a suitable universal constant $C>0$, independent of $t\in[0,T_n]$, of $n\geq1$ and of the data and solutions of our problem.

Based on the previous estimate \eqref{est:u_b-E_n_part} and on the definition of $T_n$, it is easy to see, by an induction argument, that
there exists a time $T>0$, only depending on the norms of the initial datum through $\mc E(0)$, such that
\begin{equation} \label{est:u_b-E_n_final}
 \inf_{n\geq1}T_n\,\geq\,T\,>\,0\qquad\qquad \mbox{ and }\qquad\qquad 
\sup_{n\geq1}\,\sup_{t\in[0,T]}\mc E_n(t)\,\leq\,C\,\mc E(0)\,,
\end{equation}
with the constant $C>0$ depending also on the fixed time $T$.

\subsubsection{Convergence and final checks} \label{sss:convergence}
Thanks to the uniform boundedness properties established in the previous part, we can pass to the limit in the equations and prove the convergence
of the sequence $\big(u^n,m^n\big)_{n\geq1}$ towards a true solution $\big(u,m\big)$ of the reduced micropolar system \eqref{eq:2D-microp}.

As before, let us restrict our attention to the case 
\begin{equation} \label{eq:p_finite}
 p\,\in\;]2,+\infty[\,\qquad\qquad \mbox{ and } \qquad\qquad s\,\geq\,2\,\left(1\,+\,\frac{1}{p}\right)\,,
\end{equation}
for which we will be able to apply the stability estimates of Proposition \ref{p:stab} (see also Remark \ref{r:hom-B} in this respect).
The other two cases mentioned at the beginning of Subsection \ref{ss:smooth-s} can be treated analogously, using Proposition
\ref{p:stab-energy} together with Lemma \ref{l:energy}, instead of Proposition \ref{p:stab}.

\medbreak
Let the time $T>0$ be fixed as in \eqref{est:u_b-E_n_final}. It follows that there exists a universal constant
$K>0$, only depending on the parameters $\alpha$ and $\kappa$, on the time $T$ and on the functional norms of the initial datum
$\big(u_0,m_0\big)$, such that
\begin{equation} \label{est:ub-total}
\sup_{n\geq1}\left(\left\|u^n\right\|_{L^\infty_T(B^{s}_{p,1})}\,+\,\left\|m^n\right\|_{L^\infty_T(B^{s-1}_{p,1})\cap L^1_T(B^{s+1}_{p,1})}\right)\,\leq\,K\,.
\end{equation}

We have already checked (see Paragraph \ref{sss:unique-general}, for instance) that, under our assumptions, all the hypotheses of Proposition \ref{p:stab}
are fulfilled. In particular, since $s>2$, one has the embedding
$B^s_{p,1}\,\hookrightarrow\,W^{2,\infty}$ (where we also use the fact that the space dimension is $d=2$). Thus if, for any $n\geq1$ and any $t\in[0,T]$,
we define the quantity
\[
 \mc J_n(t)\,:=\,\int^t_0\Big(1\,+\,\left\|\nabla u^n(\t)\right\|_{L^\infty}\,+\,
\left\|\nabla \o^n(\t)\right\|_{L^\infty}\,+\,\left\|\nabla m^n(\t)\right\|_{L^\infty}\Big)\,\dd\t\,,
\]
from \eqref{est:ub-total} we deduce that
\[
 \sup_{n\geq1}\,\sup_{t\in[0,T]}\mc J_n(t)\,\leq\,K\,,
\]
for a universal constant $K>0$ which may differ from the one used in \eqref{est:ub-total}. 
As a consequence, an application of Proposition \ref{p:stab} to the quantities
\[
 \de u\,=\,\de^{n,j}u\,:=\,u^{n+j}\,-\,u^n\qquad\qquad \mbox{ and }\qquad\qquad
 \de m\,=\,\de^{n,j}m\,:=\,m^{n+j}\,-\,m^n\,,
\]
defined for $n\geq1$ and $j\in\N$, immediately yields the estimate
\begin{align*}
&\left\|\delta^{n,j} u\right\|_{L^\infty_{T}(W^{1,p})}\,+\,\left\|\delta m^{n,j}\right\|_{L^\infty_{T}(L^p)}\,+\,\left\|\Delta \de m^{n,j}\right\|_{L^1_{T}(L^p)} \\
&\qquad\qquad\qquad\qquad\qquad\qquad\qquad
\leq\,C\,\Big(\left\|\delta^{n,j} u(0)\right\|_{W^{1,p}}\,+\,\left\|\delta^{n,j} m(0)\right\|_{B^{s-1}_{p,1}}\Big)\,,
\end{align*}
for a suitable universal constant $C>0$, depending on the same quantities as the constant $K$ appearing in \eqref{est:ub-total}.

From this bound and the strong convergence properties of the initial datum mentioned in Subsection \ref{ss:regul},
we immediately infer that the sequence
$\big(u^n,m^n\big)_{n\geq1}$ is a Cauchy sequence in the space $L^\infty_T\big(W^{1,p}\big)\times \Big(L^\infty_T\big(L^{p}\big)\,\cap\,L^1_T\big(W^{2,p}\big)\Big)$.
Thus, there exists a couple $\big(u,m\big)$ belonging to that functional space such that, in the limit $n\to+\infty$, one has the strong convergence
\[
 \big(u^n,m^n\big)\,{\displaystyle\longrightarrow}\,\big(u,m\big) \qquad \mbox{ in }\quad 
L^\infty_T\big(W^{1,p}\big)\times \Big(L^\infty_T\big(L^{p}\big)\,\cap \,L^1_T\big(W^{2,p}\big)\Big)\,.
\]
By uniform boundedness \eqref{est:ub-total}, the Fatou property of Besov spaces and the uniqueness of the weak limit, we discover that
\[
 \big(u,m\big)\,\in\,L^\infty_T\big(B^s_{p,1}\big)\,\times\,\Big(L^\infty_T\big(B^{s-1}_{p,1}\big)\,\cap\,L^1_T\big(B^{s+1}_{p,1}\big)\Big)\,.
\]
By interpolation, the strong convergence holds true in any intermediate space between $L^\infty_T\big(W^{1,p}\big)$  and $L^\infty_T\big(B^s_{p,1}\big)$
for $\big(u^n\big)_{n\geq1}$, and between 
$L^\infty_T\big(L^{p}\big)\,\cap\,L^1_T\big(W^{2,p}\big)$  and $L^\infty_T\big(B^{s-1}_{p,1}\big)\,\cap\,L^1_T\big(B^{s+1}_{p,1}\big)$
for $\big(m^n\big)_{n\geq1}$.

These strong convergence properties are more than enough to pass to the limit in the weak formulation of equations
\eqref{eq:def_m_n} and \eqref{eq:def_u_n}. In the end, one discovers that $\big(u,m\big)$ is indeed a solution of the original problem \eqref{eq:2D-microp}.
Finally, the time continuity properties of the solution with values in Besov spaces can be recovered by a standard inspection of the equations,
by using classical results on transport and transport-diffusion equations in Besov spaces
(notice that the target pressure can be analysed as done in Section \ref{s:a-priori}, see in particular Remark \ref{r:pressure}).

This concludes the proof of the (local in time) existence statement in Theorem \ref{th:global-strong}, under condition \eqref{eq:p_finite} over the integrability
index $p$ and the regularity index $s$.
As already mentioned, the other two cases mentioned at the beginning of Subsection \ref{ss:smooth-s} can be proved by using similar arguments.

\subsection{Local existence for less regular data} \label{ss:smooth-less}

We discuss here the construction of a solution to equations \eqref{eq:2D-microp} in the case of less regular initial data.
More precisely, we now assume to be in one of the following situations:
\begin{itemize}
 \item 
$p\in\,]2,+\infty[\,$ and $1\,+\,2/p\,\leq s\,<\,2\,\big(1\,+\,1/p\big)$;
 \item $p\in\,]1,2]$ and $s\,=\, 1\,+\,2/p$; observe that, in particular, one has $u_0,m_0\;\in L^2(\R^2)$;
\item under the assumptions of Theorem \ref{th:endpoint}, with the critical value $s= 1$.
\end{itemize}
In the very first instance mentioned above, we dispose of no stability estimates to prove convergence of the constructed scheme. Hence we must use a different strategy.
We use this same strategy also in the other cases, with critical values of $s$: the problem does not rely on the convergence of 
the approximation scheme, but rather on the possibility (or capability) to prove global estimates (see Section \ref{s:global} for more details).

First of all, for any $n\in\N$, we define smooth solutions $\big(u_n,m_n\big)$ as the solutions to the non-linear problem
\begin{equation} \label{eq:microp_n}
\left\{\begin{array}{l}
        \d_tu_n\,+\,(u_n\cdot\nabla)u_n\,+\,\nabla\Pi_n\,=\,-\,\alpha\,\nabla^\perp m_n \\[1ex]
        \d_tm_n\,+\,u_n\cdot\nabla m_n\,-\,\k\,\Delta m_n\,=\,\alpha\,\o_n \\[1ex]
        \div u_n\,=\,0\,,
       \end{array}
\right.
\end{equation}
related to the initial datum
\[
 \big(u_n,m_n\big)_{|t=0}\,=\,\big(u_{0,n}, m_{0,n}\big)\,,
\]
where the couple $\big(u_{0,n}, m_{0,n}\big)$ is the couple of smooth initial data defined in Subsection \ref{ss:regul}.

By the regularity property of these profiles (see again Subsection \ref{ss:regul} above), we can apply the local existence theory for large regularity indices
developed in Subsection \ref{ss:smooth-s}. We thus construct a sequence of smooth solutions $\big(u_n,m_n\big)_{n\in\N}$
to equations \eqref{eq:microp_n}, defined on some time intervals $[0,T_n]$.
Now, the arguments of Subsection \ref{ss:bes}, see in particular \eqref{est:Energy-unif} and \eqref{est:T-lower}, imply that
\[
 \inf_{n\in\N}T_n\,=\,T\,>\,0
\]
together with the boundedness properties
\[
 \big(u_n\big)_{n\in\N}\,\sqsubset\,L^\infty\big([0,T];B^{s}_{p,1}\big) \qquad \mbox{ and }\qquad
 \big(m_n\big)_{n\in\N}\,\sqsubset\,L^\infty\big([0,T];B^{s-1}_{p,1}\big)\,\cap\,L^1\big([0,T];B^{s+1}_{p,1}\big)\,.
\]

At this point, a standard compactness argument, based on an inspection of equations \eqref{eq:microp_n} in order to deduce strong convergence
(up to extraction of a subsequence) in suitable norms, allows one to pass to the limit and deduce the existence of
limit profiles
\[
 u\,\in\,L^\infty\big([0,T];B^{s}_{p,1}\big) \qquad \mbox{ and }\qquad
 m\,\in\,L^\infty\big([0,T];B^{s-1}_{p,1}\big)\,\cap\,L^1\big([0,T];B^{s+1}_{p,1}\big),
\]
which solve the original system \eqref{eq:2D-microp} with initial datum $\big(u_0,m_0\big)$ on $[0,T]\times\R^2$.
Finally, time continuity of $u$ and $m$ with values in the respective Besov spaces can be recovered by standard results on transport and transport-diffusion
equations (see \tsl{e.g.} Theorem \ref{th:transport}), where one also has to use the Besov regularity of the pressure gradient, mentioned in Remark \ref{r:pressure}
above.

We omit to present the details here, as the above depicted argument is fairily standard.

\section{Global existence} \label{s:global}

We prove here that the solutions constructed in the previous Section \ref{s:existence} are in fact global,
in the \emph{subcritical} regularity case $s>1+d/p$, under the assumptions
(G.1), (G.2) or (G.3) of Theorem \ref{th:global-strong}, or under the assumptions of Theorem \ref{th:endpoint}.

We are going to argue by contradiction. Assume that the lifespan of the solution is some finite $T^*>0$.
By a classical continuation argument, we get our contradiction, thus global existence, if we show that, under the above mentioned
assumptions, one has
\begin{equation} \label{est:to-prove}
 \forall\,T\in\,]0,T^*[\,\,,\qquad\qquad \sup_{t\in[0,T]}E^s_p(t)\,\leq\,C_*\,,
\end{equation}
where $E^s_p(t)$ has been defined in \eqref{eq:def_E} and for a suitable constant $C_*>0$, possibly depending on the value of the time $T^*>0$.

The starting point to prove the bound in \eqref{est:to-prove} is inequality \eqref{est:Energy_p} when $p\in\,]1,+\infty[\,$, or its counterpart
\eqref{est:Energy_inf} when $p=+\infty$. Let us recall these bounds here, in a compact form: resorting to the notation introduced in
\eqref{est:Energy-tot}, we have
\begin{equation} \label{est:En-new}
 \forall\,t\in[0,T^*[\,\,,\qquad E^s_p(t)\,\lesssim\,f_p(t)\,\left(\mc N_p(0)\,+\,
\int^t_0\mc I(\t)\,E^s_p(\t)\,\dd\t\right)\,,
\end{equation}
where the function $\mc I(t)$ has been defined in \eqref{eq:def_I}:
\[
\mc I(t)\,:=\,1\,+\,\|\nabla u(t)\|_{L^\infty}\,+\,\|u(t)\|_{L^\infty}^2\,+\,\|\o(\t)\|_{L^\infty}\,+\,\|m(t)\|_{L^\infty}\,. 
\]
Thus, we will be able to prove \eqref{est:to-prove} as soon as we will get suitable bounds for the function
$\mc I(t)$ on the time interval $[0,T^*[\,$: this is the goal of the following computations.

\medbreak
Fix some time $T$ such that $0\,<\,T\,<\,T^*$. In our computations below, in order to simplify the presentation, we agree on the following
notation: since the multiplicative constants of the bounds will depend on the time $T>0$,
we will generically denote them as time-dependent functions $K = K(t)$, whose precise expression may change from line to line, as in particular they
may depend on functional norms of the initial data. Notice that these functions $K(t)$ are increasing functions of $t\in\R_+$ and
respect the key property that $K\in L^\infty_{\loc}(\R_+)$.

Let us bound each addendum appearing in the definition of $\mc I(t)$ on $[0,T]$. First of all, owing to estimates \eqref{est:G-m_final} and \eqref{est:o_final}
when $p<+\infty$, or to estimate \eqref{est:G-m-o_inf} when $p=+\infty$, we have that
\begin{equation} \label{est:glob_o-m}
 \sup_{t\in[0,T]}\Big(\left\|\o(t)\right\|_{L^\infty}\,+\,\left\|m(t)\right\|_{L^\infty}\Big)\,\leq\,K(T)\,\leq\,K(T^*)\,.
\end{equation}

Next, we focus on the bound of the $L^\infty$ norm of the velocity field. 
We divide our discussion into various cases, depending on the value of $p$.

\paragraph*{The case of $1<p<2$, or $p=+\infty$, or assumption (G.3).}

We begin by treating the two cases $1<p<2$ or $p=+\infty$ together. As a matter of fact, by using estimate \eqref{est:u-Leb_inf}
in the latter case, or Remark \ref{r:u-Leb_p} when $1<p<2$, we deduce
\begin{equation} \label{est:glob-u}
 \sup_{t\in[0,T]}\,\left\|u(t)\right\|_{L^\infty}\,\lesssim\,K(T^*)\,.
\end{equation}

\paragraph*{The finite energy case.}
We now assume that either $p=2$, or $2<p<+\infty$ but $u_0$ and $m_0$ are of finite energy, namely they belong to $L^2$.
In order to bound the $L^\infty$ norm of $u$ in this case, we use Lemma \ref{l:GN-ineq} to get
\[
 \left\|u\right\|_{L^\infty}^2\,\lesssim\,\left\|u\right\|_{L^2}\;\left\|\o\right\|_{L^\infty}\,.
\]
Thanks to \eqref{est:glob_o-m} and to the energy estimate \eqref{est:u-m-energy}, we thus deduce \eqref{est:glob-u} again.

\paragraph*{}
Thus, it remains to bound the $L^\infty$ norm of $\nabla u$. Again, we divide our discussion into two parts:
this time, we distinguish between the subcritical case $s>1+d/p$ and the critical case $s=1+d/p$. In fact, we are able to prove global
existence only in the former instance, as we will see Subsection \ref{ss:subcrit}; in Subsection \ref{rmkCritical}, instead, we explain why our approach
fails to yield global well-posedness.

\subsection{The subcritical case: global estimates} \label{ss:subcrit}
We are going to focus for a while on the \emph{subcritical} case, namely the situation in which
\[
 s\,>\,1\,+\,\frac{2}{p}\,,\qquad\qquad p\in\,]1,+\infty]\,.
\]
Under this assumption,
we can apply a celebrated logarithmic interpolation inequality (see \tsl{e.g.} Proposition 2.104 of \cite{BCD}): we get
\[
 \left\|\nabla u\right\|_{L^\infty}\,\lesssim\,\left\|\o\right\|_{L^\infty}\,\log\left(e\,+\,\frac{\|u\|_{B^s_{p,1}}}{\|\o\|_{L^\infty}}\right)\,.
\]
Owing to the fact that the function $\z\,\mapsto\,\z\,\log\left(e\,+\,\frac{C}{\z}\right)$ is increasing over $\R_+$, no matter the value of $C>0$, thanks
also to \eqref{est:glob_o-m}, we deduce that
\begin{equation} \label{est:glob_Du}
\forall\,t\in[0,T]\,,\qquad\qquad  \left\|\nabla u\right\|_{L^\infty}\,\lesssim\,K(T^*)\,\log\Big(e\,+\,\|u(t)\|_{B^s_{p,1}}\Big)\,.
\end{equation}

At this point, plugging inequalities \eqref{est:glob_o-m}, \eqref{est:glob-u} and \eqref{est:glob_Du} into \eqref{est:En-new}, we find
\begin{equation} \label{est:En-new_II}
 \forall\,t\in[0,T^*[\,\,,\qquad E^s_p(t)\,\lesssim\,K(T^*)\,\left(\mc N_p(0)\,+\,
\int^t_0E^s_p(\t)\,\log\Big(e\,+\,E^s_p(\t)\Big)\,\dd\t\right)\,,
\end{equation}
which in particular yields the estimate \eqref{est:to-prove}, by an application of the Osgood lemma. This gives us the sought contradiction,
thus implying that the solution is global in time.

\subsection{A remark on the critical case}\label{rmkCritical}

To conclude, let us comment on the failure to prove global existence in the \emph{critical} case
\[
 s\,=\,1\,+\,\frac{2}{p}\,,\qquad\qquad p\in\,]1,+\infty]\,.
\]
Keeping in mind bounds \eqref{est:glob_o-m} and \eqref{est:glob-u}, it is clear that this is related to the control of the term
$\|\nabla u\|_{L^\infty}$ appearing in the definition of $\mathcal I(t)$.

In order to bound $\|\nabla u\|_{L^\infty}$ in the critical case, we notice that we cannot rely anymore on the logarithmic interpolation inequality crucially exploited above.
The finest strategy (in terms of needed regularity) consists in using the decomposition
\[
\nabla u\, =\, \Delta_{-1} \nabla u\, +\, \displaystyle\sum_{j=0}^{+\infty} \Delta_j {\rm BS}(\omega)\,,
\]
where we have written $\nabla u = {\rm BS}(\omega)$, according to the Biot-Savart law \eqref{eq:BS}, and we have isolated the low frequency term $\Delta_{-1} \nabla u$.
In the global estimate for $\nabla u$, the low frequency term $\Delta_{-1}\nabla u$ is usually bad, as the singularity of the Biot-Savart law is concentrated precisely at
low frequencies.
However, after the frequency splitting, thanks to Lemma \ref{l:bern} the low frequency term can be bounded by \tsl{e.g.} the $L^p$ norm of $\omega$,
which enjoys global estimates, as shown in \eqref{est:glob_o-m}.
Nonetheless, in this way bounding the high frequency terms inevitably yields a stronger norm: one finds
\[
 \left\|\nabla u\right\|_{L^\infty}\,\lesssim\,\|\omega\|_{L^p}\,+\,\left\|\o\right\|_{B^0_{\infty,1}}\,.
\]
Therefore, the problem is reconducted to the bound of $\o$ in the critical space $B^0_{\infty,1}$.
Now, using the same strategy described in Subsection \ref{ss:leb} and taking advantage of improved transport estimates \cite{HK, Vis}
in Besov spaces $B^0_{p,r}$, we can arrive at an estimate of the form
\[
 \left\|\o(t)\right\|_{B^0_{\infty,1}}\,+\,\left\|m(t)\right\|_{B^0_{\infty,1}}\,\lesssim\,
\left(\mc N_p(0)\,+\,\int^t_0\left\|\o(\t)\right\|_{B^{0}_{\infty,1}}\,\dd\t\right)\,\left(1\,+\,\int^t_0\left\|\nabla u(t)\right\|_{L^{\infty}}\,\dd\t\right)\,.
\]
This estimate is not of very practical use, and in fact it indicates that $\left\|\o(t)\right\|_{B^0_{\infty,1}}$ satisfies
an estimate of the kind
\[
 \eta'(t)\,\lesssim\,\left(1\,+\,\int^t_0\eta(\t)\,\dd\t\right)^2\,,
\]
which rather gives only a finite time control. So, this strategy does not work.

As a last comment, we point out that, thanks to the strategy of construction of approximate solutions explained in Subsection \ref{ss:smooth-less}
and exploiting the global bounds of the Lebesgue norms (see Subsection \ref{ss:leb}), in the critical case we can construct global solutions
at a slightly lower level of regularity. For instance, when $p=+\infty$ we can use the result of \cite{F-FD} to get solutions
having, roughly speaking, $u\in \mc C\big(\R_+;W^{1,q_0}(\R^2)\big)$ and $m\in \mc C\big(\R_+;L^{p_0}(\R^2)\big)\cap L^2_\loc\big(\R_+;L^2(\R^2)\big)$,
with $q_0$ defined in Proposition \ref{p:BS}; we can proceed similarly in the other cases
$p\in\,]1,+\infty[$ and $s=1+2/p$ (with, possibly, a finite energy condition on the initial datum).
These solutions are also unique, in light of the statement of Proposition \ref{p:uni-energy}. However, what is missing is the
propagation the critical Besov norm for all times. The strategy presented above, inspired by \tsl{e.g.} \cite{A-H-K, HK}, fails to yield a global control
on such a norm, essentially because of the presence of the (linear!) forcing terms on the right-hand side of equations \eqref{eq:2D-microp}.


%


\end{document}